\documentclass[reqno,12pt]{amsart}
\usepackage{enumerate}
\usepackage[margin=0.77in]{geometry}
\usepackage{amsmath,amssymb,amsthm}
\usepackage{dsfont}
\usepackage{graphicx}
\usepackage{microtype}

\usepackage[dvipsnames]{xcolor}
\usepackage{hyperref}
\hypersetup{
    colorlinks=true,
    linkcolor=cyan!80!black,
    citecolor=MidnightBlue,
    urlcolor=magenta,
}

\theoremstyle{plain}
\newtheorem{theorem}{Theorem}

\newtheorem{proposition}{Proposition}
\newtheorem{lemma}{Lemma}
\newtheorem{remark}{Remark}

\theoremstyle{definition}

\newtheorem{definition}{Definition}

\DeclareMathOperator{\Cov}{Cov}

\newcommand{\wt}{\mathsf{wt}}
\newcommand{\supp}{\mathsf{supp}}

\let\tilde\widetilde

\title{Operator limit of Wigner matrices I}
\author{Debapratim Banerjee}
\email{debapratim.banerjee@ashoka.edu.in}
\address{
    Department of Mathematics \\
    Ashoka University\\
    Plot no 2, Rajiv Gandhi Education City, Sonipat 131029\\
    Haryana, India.
}
\begin{document}

\maketitle
\begin{abstract}
    We consider the Wigner matrix $W_{n}$ of dimension $n \times n$ as $n \to \infty$. The objective of this paper is two folds: first we construct an operator $\mathcal{W}$ on a suitable Hilbert space $\mathcal{H}$ and then define a suitable notion of convergence such that the matrices $W_{n}$ converge in that notion of convergence to $\mathcal{W}$. We further investigate some properties of $\mathcal{W}$ and $\mathcal{H}$. We show that $\mathcal{H}$ is a nontrivial extension of $L^{2}[0,1]$ with respect to the Lebesgue measure and the spectral measure of $\mathcal{W}$ at any function $f \in L^{2}[0,1]$ is almost surely the semicircular law.
\end{abstract}
\section{Introduction}
The study of random matrices began with the seminal papers of Wigner \cite{wig1} and \cite{wig2}. He considered symmetric matrices of growing dimension where the entries in the upper diagonal part are i.i.d. mean zero and variance one. In these papers, he determined the limit of histograms of the eigenvalues of such matrices under proper scaling. This limit is famously known as the semicircular law. As time passed, several finer questions came into consideration and were solved. In modern times, we have seen remarkable progress in the study of random matrices, in general. For example, the successful resolution of the Dyson-Mehta conjectures can be seen as a cornerstone of these achievements. These show many strong results on the eigenvalues of such matrices. One might look at \cite{mehta2004random,erdHos2010bulk, MR2661171, MR2851058, erdHos2012local,erdHos2012bulk, erdHos2012rigidity}. The breadth and depth of the current results are so great that we shall not be able to discuss them in this introduction. However, to the best of our knowledge, there is no paper that attempts to define a limiting operator of such matrices. As Wigner matrices were invented to study the energy of the nuclei of heavy atoms, defining a limiting operator might be useful from the physics point of view. The series of works which attempted to define an operator limit of random matrices is by Virag and co-authors. One might look at \cite{ramirez2011beta} and the follow-up papers. These papers consider GOE/GUE matrices and define the operator limits of  tridiagonalized versions of these matrix. In particular, following \cite{ramirez2011beta} section 2, these operators are called stochastic Airy operators. The operator $\mathcal{H}_{\beta}$ ($\beta =1$ for GOE, $\beta =2$ for GUE) is a random linear map from $H^{1}_{\mathrm{Loc}} \to D$ and is given by 
\[
\mathcal{H}_{\beta}(f)= -f''+ xf + \frac{2}{\sqrt{\beta}} b'.
\]
Here $D$ is the space of distributions,$H^{1}_{\mathrm{Loc}}$ is the space of all measurable functions $f: \mathbb{R}^{+} \to \mathbb{R}$  for which $f'\mathbb{I}_{I} \in L^{2}$ for any compact set $I$ and $b$ is the standard Brownian motion. Any symmetric matrix can be unitarily tridiagonalized. Hence, tridiagonalization although preserves the eigenvalues, the information about eigenvectors are lost.

In this article, we construct a limiting operator of the Wigner matrices without tridiagonalizing it. The major challenge remains in construction of the space on which the limiting operator is defined. Recalling the proof of the spectral theorem of self-adjoint bounded operators on Hilbert spaces (\cite{bhatia2009notes} for example), a major assumption is, the range and domain of the operator should be the same. This assumption is really crucial for eigenvalues and can be seen on finite dimensional Hilbert spaces as well. For example we consider the matrix $X=\left[\begin{array}{cc}
   0  & 1 \\
    1 & 0
\end{array}\right].
$
This matrix has eigenvalues $\pm 1$ when it is defined from $\mathbb{R}^2 \to \mathbb{R}^2$. However, if we define the operator on the one dimensional subspace $\left( \begin{array}{c}
     c \\
     0 
\end{array} \right)$, then  $\left[\begin{array}{cc}
   0  & 1 \\
    1 & 0
\end{array}\right] \left( \begin{array}{c}
     c \\
     0 
\end{array} \right)= \left(\begin{array}{c}
     0  \\
     c 
\end{array} \right)$ implying no eigenvalues. We face this kind of issues while defining the limiting operators. First of all, we consider any of dimension $n \times n$ as a kernel operator $K_{n}$ on $L^{2}\{\frac{1}{n},\ldots, \frac{n}{n}  \}$ with respect to the measure $\mu_{n}= \frac{1}{n}\left( \sum_{i=1}^{n} \delta_{\frac{1}{n}} \right)$ so that our limiting space is an approximation of the function space $L^2[0,1]$. However the major difficulty remains in determining the limit of $K_{n}(f)$ for $L^{2}$ functions. For example, even for the constant function $f=1$, we find the limit of $K_{n}(f)$ is not a Lebesgue measurable function. Hence, if we apply the operator to any measurable function, the image goes out of the space of measurable functions. In section \ref{sec:nonmeasure} we build enough machinery to deal with such objects.  This creates a problem defining the limiting operator. To overcome this, we consider $K_{n}^{l}(\mathbb{I}_{P_{i}})$ for every $l$ and intervals $P_{i}$ with rational endpoints so that image of the operator doesn't go out side the space. Actually, the space we work with turns out to be a Hilbert space which is a non-trivial extension of the $L^{2}[0,1]$. Further, the limit of $K_{n}(f)$ for any measurable function $f$ is orthogonal to any measurable function.  A majority of the paper is devoted to formalize such vector space. 

The definition of the limiting space, the limiting operator and the operator convergence is given in Definitions \ref{def:limspace} \ref{def:limop} and \ref{def:opconv} respectively. Before moving forward, we give some connection to the existing notions of operator convergence. In particular, we compare it with strong and weak operator convergence. In both notions of convergence, the space on which the operators are defined does not change. On the other hand, in our case, the space changes with $n$. In particular, for every $n$, we have a kernel acting on $L^{2}\{\frac{1}{n},\ldots, \frac{n}{n}\}$. We prove that for any interval $P_{i}=(a,b]$, $K_{n}^{l}(\mathbb{I}_{P_{i}})=\beta_{l}\mathbb{I}_{P_{i}}+ \psi_{n,l,i}+ \varepsilon_{n,l,i}$ such that $\varepsilon_{n,l,i}$ is a vector of norm close to zero and $ \psi_{n,l,i} \to \psi_{l,i}$ in some notion such that $\langle \psi_{n,l_{1},i_{1}},\psi_{n,l_{2},i_{2}} \rangle \to \langle \psi_{l_{1},i_{1}}, \psi_{l_{2},i_{2}} \rangle $. The limiting operator is defined in such a way that  $\langle v_{n,1}, K_{n}(v_{n,2}) \rangle \to \langle v_{1}, \mathcal{W}(v_{2}) \rangle$. Here $v_{1}$ and $v_{2}$ are the limits of $v_{n,1}$ and $v_{n,2}$ respectively. In this sense, this is closely related to weak operator convergence. 

Finally, regarding the spectral properties of the limiting operator, we mention that the spectral measure at any $L^{2}$ measurable function is the semicircular law. This is the coarsest resolution of eigenvalues that we can look at. We believe that there is a hierarchy of operators depending on the resolution of eigenvalues we want to examine. For example, we believe there is a different operator if we want to capture the edge universality. We leave this for future work.  
\section{Notations and organizations}.
Throughout this paper, we shall follow the following notations: 
\begin{enumerate}
    \item The Wigner matrices of dimension $n \times n$ will be denoted by $W_{n}$.
    \item $\beta_{l}$ will be used to denote the $l^{th}$ moment of the semicircular law. In particular, $\beta_{2l+1}=0$ and $\beta_{2l}= \frac{1}{l+1}\binom{2l}{l}$.
    \item $\mu_{n}$ will be used to define the measure $\frac{1}{n} \sum_{i=1}^{n} \delta_{\frac{i}{n}}$.
    \item The limit of the Wigner matrices will be denoted by $\mathcal{W}$ and the Hilbert space on which $\mathcal{W}$ is defined will be denoted by $\mathcal{H}$.
\end{enumerate}
The paper has the following four parts. In the first part, we consider the Wigner matrices as kernel and argue that if we want to define the limit of the operator, then the space must be a nontrivial extension of the $L^{2}[0,1]$ space. In the second part, we develop the theory for the vectors that we need to include to define the limiting operator, but are outside the space $L^{2}[0,1]$. In the third part, we define the notion of convergence and define the limiting Hilbert space $\mathcal{H}$ and $\mathcal{W}$ formally. In the final part, we prove that $W_{n}$ converges to $\mathcal{W}$ in that notion of convergence and the spectral measure of $\mathcal{W}$ at any $f\in L^{2}[0,1]$ is almost surely the semicircular law.

\section{Wigner matrices}\label{sec:wig}
 In this section, we introduce the matrix ensembles. Firstly, we start with the definition of Wigner matrices. 
 \begin{definition}\label{def:wig}
 We call a matrix $W_{n}= \left( x_{i,j}/ \sqrt{n} \right)_{1\le i,j \le n}$ to be a Wigner matrix if $x_{i,j}=\bar{x}_{j,i}$, $\left(x_{i,j}\right)_{1\le i<j\le n}$ are i.i.d., $\mathrm{E}[x_{i,j}]=0$ and $\mathrm{E}[|x_{i,j}|^2]=1$. 
 \end{definition}
 In this paper, we define the limit of this matrix. Special cases of Wigner matrix are known as the GOE/GUE matrix where the entries are real(GOE)/complex(GUE) Gaussian random variables.
 
As mentioned in the introduction, our main approach is to scale the set $\{ 1,\ldots ,n \}$ by $n$ and view this as an approximation of the interval $[0,1]$. To do this, we view the Wigner matrix as a kernel operator on the space $L^{2}\{ \frac{1}{n},\ldots, \frac{n}{n}  \}$ with respect to the measure $\mu_{n}= \frac{1}{n} \left( \sum_{i=1}^{n} \delta_{\frac{i}{n}} \right)$ such that the eigenvalues remain unchanged. 

\subsection{Wigner matrices as kernel operators on $L^{2}\{ \frac{1}{n},\ldots, \frac{n}{n}  \}$}\label{kernel}  \
\begin{definition}\label{def:kernel}
A kernel operator $\left(K\left(\frac{i}{n},\frac{j}{n}\right)\right)_{1\le i,j \le n}$ is a function on $\left\{  
\frac{1}{n},\ldots, \frac{n}{n} \right\}^2$ such that for any function $f$ on $L^{2}\left\{ \frac{1}{n},\ldots, \frac{n}{n}  \right\}$ with respect to $\mu_{n}$ we have
\[
K(f)\left(\frac{i}{n}\right)=  \frac{1}{n}\sum_{i=1}^{n} K\left(\frac{i}{n},\frac{j}{n}\right) f\left(\frac{j}{n}\right).
\]
\end{definition}
\begin{lemma}\label{lem:wigtokernel}
Let $W_{n}$ be the Wigner matrix as defined in Definition \ref{def:wig} and $K_{n}$ be the  kernel operator on $L^{2}\left\{ \frac{1}{n}, \ldots, \frac{n}{n} \right\}$ defined as  
\[
K_{n}\left( \frac{i}{n}, \frac{j}{n} \right)= \sqrt{n} x_{i,j}.
\]
Then $W_{n}$ and $K_{n}$ have the same eigenvalues.
\end{lemma}
\begin{proof}
Let $\lambda$ be an eigenvalue of $W_{n}$. Then we know that there is a vector $g=(g_{1},\ldots, g_{n})$ in $\mathbb{R}^{n}$ such that $||g||^2=1$ and $W_{n}g =\lambda g$. Now $|| g ||^2= \sum_{i=1}^{n}g_{i}^2$. We construct $f \in L^{2}\left\{ \frac{1}{n},\ldots, \frac{n}{n} \right\}$ with respect to $\mu_{n}$ such that $f(\frac{i}{n})=c g(i)$ with $c>0$ and $||f||^2=||g||^2$. We know that 
\[
||f||^{2}= \frac{1}{n}\sum_{i=1}^{n} f\left( \frac{i}{n} \right)^2= \frac{c^2}{n} \sum_{i=1}^{n} g(i)^2= \frac{c^2}{n}.
\]
Hence, $c= \sqrt{n}$. Now, 
\begin{equation}
\begin{split}
K_{n}(f)\left( \frac{i}{n} \right)& = \frac{1}{n} \sum_{j=1}^{n} K_{n}\left( \frac{i}{n}, \frac{j}{n} \right) f\left(\frac{j}{n}\right)\\
& = \frac{1}{n}\sum_{j=1}^{n} \sqrt{n} x_{i,j} \sqrt{n} g(j)\\
&= \lambda \sqrt{n} g(i)= \lambda f\left(\frac{i}{n}\right).
\end{split}
\end{equation}
The opposite can also be proved in the similar way. Hence $W_{n}$ and $K_{n}$ have the same eigenvalues.
\end{proof}
As a consequence, we shall apply Lemma \ref{lem:wigtokernel} and work with $K_{n}(\cdot,\cdot)$ instead of $W_{n}$.
\section{The extension of $L^{2}[0,1]$}\label{sec:nonmeasure}
In this paper, we see the space $\{\frac{1}{n},\ldots, \frac{n}{n}  \}$ as an approximation of the space $[0,1]$.  
Our main goal is to define  some kind of limit of the kernel operator $K_{n}(\cdot,\cdot)$ defined in Lemma \ref{lem:wigtokernel} as $n$ becomes larger and larger. We want to find a space such that the image under the action of the limiting operator (provided it exists) is contained in the space. However, in this section, we prove that the space $L^{2}[0,1]$ can not be candidate for this space. In particular, the action of $K_{n}$ on constant functions gives rise to vectors which in the limiting sense can't be measurable functions.

We first give an example regarding this. For the time being, we assume that $x_{i,j}$' s are i.i.d. standard Gaussians and we write $K_{n}= \frac{1}{\sqrt{2}}\left(K_{n,1}+ K_{n,2}\right)$ where $K_{n,1}\left(\frac{i}{n}, \frac{j}{n}\right)= \sqrt{n} y_{i,j}$ and $K_{n,2}\left(\frac{i}{n}, \frac{j}{n}\right)= K_{n,1}\left( \frac{j}{n}, \frac{i}{n} \right)$ with $y_{i,j}$'s are i.i.d. standard Gaussians. To avoid calculations hassles (which will be considered in the main part of the paper), we shall act $K_{n,1}$ on the constant vector $f_{n}=(1,\ldots,1)$ that has $L^2$ norm $1$ as an element in $L^2\{\frac{1}{n},\ldots, \frac{n}{n} \}$ with respect to $\mu_{n}$.

Applying $K_{n,1}$ on $f$, we get 
\begin{equation}
\begin{split}
K_{n,1}(f)\left( \frac{i}{n} \right)&= \frac{1}{n} \sum_{i=1}^{n} K_{n,1}\left(\frac{i}{n}, \frac{j}{n}\right)f\left( \frac{j}{n} \right)\\
&= \sum_{j=1}^{n}\frac{1}{\sqrt{n}} y_{i,j}\\
&= \eta(i)
\end{split}
\end{equation}
with $\eta(i)$'s are i.i.d. $N(0,1)$. Observe that 
\[
||K_{n,1}(f) ||^2= \frac{1}{n}\sum_{i=1}^{n} \eta(i)^2 \stackrel{a.s.}{\to} 1 
\]
as $n \to \infty$. So, approximately the limit of $K_{n,1}(f)$ is a unit vector. However, the main problem arises in the measurability of the limit of $K_{n,1}(f)$. In particular, the way we define the limit of $K_{n,1}(f)$ will be orthogonal to any step function (random or nonrandom). Hence, the limiting object is not a measurable function. We identify the limit of such vectors as $\frac{dB(x)}{\sqrt{dx}}$. This notion is formalized in the next subsection.
\subsection{Defining the limiting objects}\label{subsec:delta}
We have briefly mentioned that we need to deal with vectors where the limits are not measurable functions. In this section, we give a unified way to deal with these vectors. This treatment is applicable to general Dirac $\delta$ functions and white noises as well. This treatment is quite similar to classical treatment of delta functions and white noises but it includes objects like the limit of $K_{n,1}(f)$ as $n$ goes to infinity.

\noindent 
As a starting point, we start with the Dirac $\delta$ function at some point $x_{0}$, which we denote $\delta_{x_{0}}(x)$. We know that heuristically $\delta_{x_{0}}(x)$ can be thought of $\frac{dF_{x_{0}}(x)}{dx}$ where $F_{x_{0}}(x)$ is defined as follows:
\begin{equation}
F_{x_{0}}(x)= \left\{
\begin{array}{ll}
0 & \text{whenever $x < x_{0}$}\\
1 & \text{otherwise.}
\end{array}
\right.
\end{equation}
However, the main problem here is that the function $F_{x_{0}}(x)$ is not differentiable at the point $x_{0}$. On the other hand, from measure theory $F_{x_{0}}(x)$ can be viewed as the distribution function of the probability measure that gives unit mass to the point $x_{0}$. We call this measure $\mu_{x_{0}}$. In particular, we have 
\[
\mu_{x_{0}}((a,b])= F_{x_{0}}(b)- F_{x_{0}}(a).
\] 
Observe that $\mu_{x_{0}}$ is a function on sets rather than points.
On the other hand, for any differentiable function $F$, we know that 
\[
\frac{dF(x)}{dx}= \lim_{h \to 0} \frac{F(x+h)-F(x)}{h}.
\]
Combining these two concepts, we view the derivatives as follows:
\begin{definition}\label{def:derset}
Suppose $F$ is a function which is not necessarily differentiable. Then the derivative of $F$ is defined as a set function on non-trivial  half open intervals and is given by 
\[
\frac{dF}{dx}((a,b])= \frac{F(b)-F(a)}{b-a}
\]
for any $-\infty < a<b <\infty.$
\end{definition}
Observe that if $F$ is differentiable at a point $x$, this set function converges to $F'(x)$ as $a \to b$. So, this definition is consistent with the classical definition of derivatives. However, this allows us to give an interpretation of the white noise which is defined as the derivative of the standard Brownian motion. 

It is a well-known fact that if $B(x)$ is the standard Brownian motion, then $B(b)-B(a) \sim O(\sqrt{b-a})$ for $a<b$. In this paper, we shall show that the standard Brownian motion $B(x)$ is $\frac{1}{2}$ Holder continuous in the $L^{2}$ sense. Definition  \ref{def:derset} is flexible enough to define $\frac{dF}{(dx)^{\alpha}}$ for any $0< \alpha <1$. It is defined as follows:
\begin{definition}\label{def:dersetalpha}
Let $F$ be a function. Then, for any $0<\alpha <1$, we define:
\[
\frac{dF}{(dx)^{\alpha}}((a,b])= \frac{F(b)-F(a)}{(b-a)^{\alpha}}.
\]
\end{definition}
Hence, we can also define $\frac{dB(x)}{\sqrt{dx}}$. We shall later identify the limit of $K_{n,1}(f)$ as $\frac{dB(x)}{\sqrt{dx}}$. 
\begin{definition}
Using this definition, it is also straight forward to define the inner product of $F$ and $G$ where $F$ and $G$ are set functions on half open intervals in the following way:
\begin{equation}\label{def:inner}
\int_{0}^{1} F(x)G(x) dx= \lim_{||\mathcal{P}|| \to 0} \sum_{j=1}^{l} G((a_{j-1},a_{j}])F((a_{j-1},a_{j}]) (a_{j}-a_{j-1}) ~ \text{(provided it exists)}
\end{equation}
\end{definition}

\begin{remark}\label{rem:measurable}
    Given any measurable function $f \in L^{2}[0,1]$ with respect to the Lebesgue measure, the natural analogue of $f$ as a set function is given by 
    \[
    \tilde{f}((a,b])= \frac{\int_{a}^{b} f(x) dx}{(b-a)}.
    \]
    In Propositions \ref{prop:well}, we shall prove that $\tilde{f}$ is well defined, i.e., if two functions $f$ and $g$ give rise to the same set functions, then $f=g$ (a.s. w.r.t. Lebesgue measure). Furthermore, in Proposition \ref{prop:inner}, we shall prove that for two functions $f$ and $g$, $\langle f,g \rangle= \langle \tilde{f}, \tilde{g} \rangle $, where $\tilde{f}$ and $\tilde{g}$ are the set functions corresponding to $f$ and $g$. Here, $\langle \tilde{f}, \tilde{g}\rangle$ is as defined in \eqref{def:inner}
   
\end{remark}
\begin{remark}
    Given a set function $f$, we would ideally like to define
    \begin{equation}\label{eq:integral}
    \int_{a}^{b} f(x) (dx)^{\alpha}= \lim_{||\mathcal{P}||\to 0} \sum_{j=1}^{l} f((a_{j-1},a_{j}])(a_{j}-a_{j-1})^{\alpha} ~\text{(provided it exists)}
    \end{equation}
    for a partition $\mathcal{P}= \{ a_{0}=a< a_{1}<\ldots<a_{l}=b\}$ and $||\mathcal{P}||= \max_{j}(a_{j}-a_{j-1})$. It is easy to see that $\int_{a}^{b} \frac{df(x)}{(dx)^{\alpha}} (dx)^{\alpha}= f(b)-f(a)$.
    However, there are some issues one needs to be careful about. According to Remark \ref{rem:hausdorff}, the existence of the limit in the r.h.s. of \eqref{eq:integral} requires a lot of cancellations. Hence, many set functions which have $0$ $L^2$ norm, would not have a finite integral. These issues have been carefully addressed in Remarks \ref{rem:intequi} and \ref{rem:limitspace}. One might look at these. 
\end{remark}
\begin{remark}\label{rem:hausdorff}(Comparison with Hausdorff measures )
One might note that a similar concept exists for the definition of Hausdorff measures. For example, for any $\alpha >0$, an outer measure of a set $S$ is constructed in the following way. Let $\{ U_i\}_{i=1}^{\infty}$ be a countable cover of $S$ such that $\mathrm{diam}(U_{i})\le \delta$. Then 
\[
H_{\alpha}(S)= \lim_{\delta \to 0} \inf \left\{ \sum_{i=1}^{\infty} \left( \mathrm{diam}(U_{i}) \right)^{\alpha}  \right\}
\] 
In particular, taking $U_{i}$'s to be half open intervals and $g$ to be identically $1$ one gets similar concepts. However, the function $g=1$ is not integrable over any interval of the form $(a,b]$ as it's value is $\infty$. In particular, in order to get a finite value over an interval $(a,b]$, one needs a lot of cancellations like the quantity $\frac{dB(x)}{\sqrt{dx}}$. In general, one will not get a finite value over an interval by integrating standard functions $g$ with respect to $(dx)^{\alpha}$.
\end{remark}


\begin{remark} 
We would like to mention that not every set function will be a derivative. In particular, for $g= \frac{dF}{(dx)^{\alpha}}$, we need the set function $g((a,b])(b-a)^{\alpha}$ to be $\sigma$ additive.
\end{remark}

\noindent 
In this paper, we shall only need derivative with respect to $\sqrt{dx}$. So, from now on, we shall only work with $\frac{dF(x)}{\sqrt{dx}}$. 

\noindent 
\textbf{Equivalence classes of set functions:}\\
From Definition \ref{def:dersetalpha}, it might appear that we can define $\frac{dF(x)}{(dx)^{\alpha}}$ for any function. However, we need to be aware of some issues. Actually, the derivative with respect to $(dx)^{\alpha}$ gives a measurement of smoothness. For example, $\frac{dB(x)}{\sqrt{dx}}$ measures how much ``smooth" the Brownian motion is with respect to $\sqrt{dx}$. Now, consider a differentiable function $F(x)$. This has a higher order smoothness than $\frac{1}{2}$ Holder continuity. Indeed, if we define 
\[
g(x)= \lim_{h \to 0}\frac{F(x+h)- F(x)}{|h|^{\frac{1}{2}}},
\]
then this limit exists and is equal to $0$. In other words, we should get $\frac{dF(x)}{\sqrt{dx}}=0$. This is analogous to the case $\frac{dF(x)}{dx}=0$ whenever $F(x)$ is a constant function (constant functions have higher order of smoothness than differentiable functions). This phenomenon should reflect in the set function definition and it indeed does. In particular, if we take function $F$ with $0$ quadratic variation and consider the set functions $\frac{dF(x)}{(dx)^{\frac{1}{2}}}$ as in Definition \ref{def:dersetalpha}, then 
\begin{equation}
    \begin{split}
        &\int_{0}^{1} \left(\frac{dF(x)}{\sqrt{(dx)}}\right)^2 dx \\
        &= \lim_{||\mathcal{P}||\to 0} \sum_{j=1}^{l} \left( F(a_{j}) - F(a_{j-1})\right)^2 =0.
    \end{split}
\end{equation}
So in particular, $\frac{dF(x)}{\sqrt{(dx)}}$ is a vector which has $0$ $L^{2}$ norm. This forces us to define equivalence classes of set functions in the following way.
\begin{definition}\label{def:equivalence}
    We consider two set functions $F$ and $G$.
    We say $F$ and $G$ are equivalent if
    \[
      \int_{0}^{1} (F(x)-G(x))^2 dx= \lim_{||\mathcal{P}||\to 0} \sum_{j=1}^{l} \left(F((a_{j-1},a_{j}])-G((a_{j-1},a_{j}])\right)^{2} |a_{j}-a_{j-1}|= 0. 
    \]
\end{definition}
It can be easily checked that this is indeed an equivalence class.
\begin{remark}\label{rem:intequi}
 Observe that all set functions $F(x)$ such that $\int_{0}^{1} F^{2}(x)dx=0$ will not be integrable with respect to $\sqrt{dx}$. We have argued in Remark \ref{rem:hausdorff} that if the limit 
 \[
\lim_{||\mathcal{P}||\to 0} \sum_{j=1}^{l} F((a_{j-1},a_{j}]) \sqrt{a_{j}-a_{j-1}}
 \]
 exists, then we need a lot of cancellations.  In particular, we can have set functions $G$ such that $\int_{0}^{1} G(x)^{2}dx=0$ but still $\int_{0}^{1}G(x)\sqrt{dx}$ does not make sense.
One might consider the following example. Take $G((a,b])=(b-a)^{\frac{1}{4}}$. Observe that
\begin{equation}
\begin{split}
 &\lim_{||\mathcal{P}||\to 0} \sum_{j=1}^{l}G^{2}((a_{j-1},a_{j}])(a_{j}-a_{j-1})\\
 & \lim_{||\mathcal{P}||\to 0} \sum_{j=1}^{l} (a_{j}-a_{j-1})^{\frac{3}{2}} =0.
\end{split}
\end{equation}
but 
\[
\lim_{||\mathcal{P}||\to 0} \sum_{j=1}^{l}G((a_{j-1},a_{j}])(a_{j}-a_{j-1})^{\frac{1}{2}} = \lim_{||\mathcal{P}||\to 0}\sum_{j=1}^{l} (a_{j}-a_{j-1})^{\frac{3}{4}}
\]
doesn't exists. These set functions are equivalent to the $0$ vector. Let $F(x)$ be a function that has $0$ quadratic variation. Then any set function $G$ which is equivalent to $\frac{dF(x)}{\sqrt{dx}}$ has:
\[
\int_{0}^{1} \left( G(x) \right)^{2} dx= 0.
\]  
\end{remark}
 

\noindent
\textbf{Representing measurable functions as set functions:}\\
\begin{proposition}\label{prop:zero}
    Consider $G$ to be a set function on intervals. Let $\int_{0}^{1} G(x)^{2} dx =0$. Then for any $a,b \in [0,1]$, $\int_{a}^{b}G(x)dx=0$.
\end{proposition}
\begin{proof}
    This result is a simple use of Cauchy-Schwarz inequality. We take any $a,b \in [0,1]$. Since $\int_{0}^{1} G(x)^{2}dx = 0$, we have 
    \[
    \lim_{||\mathcal{P}||\to 0} \sum_{j=1}^{l} G^2((a_{j-1},a_{j}])(a_{j}-a_{j-1}) = 0.
    \]
    Here, $\mathcal{P}$ is a partition $\{ a_{0}=a< a_{1}< \ldots <a_{l}=b \}$. Now by Cauchy-Schwarz inequality, 
    \begin{equation}\label{eq:cauchy}
        \begin{split}
            &\left(\sum_{j=1}^{l} \left|G((a_{j-1}-a_{j}])\right|\sqrt{(a_{j}-a_{j-1})}\sqrt{(a_{j}-a_{j-1})}\right)^2\\
            & \le \left( \sum_{j=1}^{l} G^{2}((a_{j-1},a_{j}]) (a_{j}-a_{j-1}) \right)\left(\sum_{j=1}^{l}(a_{j}-a_{j-1})\right)\\
            &=(b-a) \left( \sum_{j=1}^{l} G^{2}((a_{j-1},a_{j}]) (a_{j}-a_{j-1}) \right).
        \end{split}
    \end{equation}
    Hence, the limit in L.H.S. of \eqref{eq:cauchy} goes to $0$ as $||\mathcal{P}||\to 0$. This completes the proof.
\end{proof}
Given a measurable function $f \in L^{2}[0,1]$, we consider the set function $\tilde{f}((a,b])= \frac{\int_{a}^{b}f(x)dx}{(b-a)}$.
The following two Propositions prove that there is no problem defining measurable functions as set functions. Suppose, we consider two measurable functions $f$ and $g$ such that the set functions $\frac{\int_{a}^{b}f(x)dx}{(b-a)}, \frac{\int_{a}^{b}g(x)dx}{(b-a)}$ are equivalent. Then Proposition \ref{prop:well} proves that they are almost surely same. By Proposition \ref{prop:zero}, it is enough to consider the case 
$\frac{\int_{a}^{b}f(x)dx}{(b-a)}= \frac{\int_{a}^{b}g(x)dx}{(b-a)}.$ 
\begin{proposition}\label{prop:well}
Let $f$ and $g$ be two functions in $L^2[0,1]$ with respect to the Lebesgue measure such that they give rise to the same set function $\tilde{f}$, i.e., $\int_{a}^{b} f(x) dx= \int_{a}^{b} g(x) dx$ for all $a,b \in [0,1]$. Then $f= g$ almost surely with respect to the Lebesgue measure.
\end{proposition}
\begin{proof}
    We prove this by contradiction. Suppose $f \neq g$ almost surely. Let $A= \{ f >g \}$ and $B= \{ f< g\}$. Under the hypothesis, we have $\max\{ \mu(A), \mu(B) \} >0$. Without loss of generality, assume that $\mu(A)>0$. Now $A= \cup_{n \in \mathbb{N}}\{ f> g+ \frac{1}{n} \}$. So, there exists $\varepsilon >0$ and $n_{\varepsilon} \in \mathbb{N}$ such that $\mu(\{ f> g+ \frac{1}{n_{\varepsilon}} \}) > \varepsilon$. Let us define $D:=\{ f> g+ \frac{1}{n_{\varepsilon}} \}$. $D$ is Lebesgue measurable. So, there is $S$ such that $S= \cup_{i=1}^{m}(a_{i},b_{i}]$ and $\mu(S \triangle D) < \frac{\varepsilon}{\kappa}$. Here, the constant $\kappa$ will be chosen later. Now, 
    \[
    \int_{S} f(x) dx= \sum_{i=1}^{m} \tilde{f}((a_{i},b_{i}])(b_{i}-a_{i})= \int_{S} g(x)dx.
    \]
    However, 
    \begin{equation}
        \begin{split}
          \int_{S} f(x) dx&= \int_{D\cap S} f(x) dx + \int_{S \backslash D} f(x) dx\\
          &\ge \int_{D\cap S} g(x) dx + \frac{1}{n_{\varepsilon}} \mu(D\cap S) - \left| \int_{S\backslash D} f(x) dx\right|\\
          & \ge \int_{D\cap S} g(x) dx + \frac{\varepsilon}{n_{\varepsilon}}\left(1- \frac{1}{\kappa} \right) - ||f||_{2}(\mu(S \backslash D))^{\frac{1}{2}}\\
          & \ge \int_{S} g(x) dx + \frac{\varepsilon}{n_{\varepsilon}}\left(1- \frac{1}{\kappa} \right) - \left(||g||_{2}+||f||_{2}\right) \sqrt{\frac{\varepsilon}{\kappa}}.
        \end{split}
    \end{equation}
    Now, we choose $\kappa$ large enough so that 
    \[
      \frac{\varepsilon}{n_{\varepsilon}}\left(1- \frac{1}{\kappa} \right) \ge 2 \left(||g||_{2}+||f||_{2}\right) \sqrt{\frac{\varepsilon}{\kappa}}.
    \]
    We can do this as all the other parameters are previously fixed. This gives 
    \[
    \int_{S} f(x) dx \ge \int_{S} g(x) dx + \frac{\varepsilon}{2n_{\varepsilon}}\left(1- \frac{1}{\kappa} \right).
    \]
    This is impossible.
\end{proof}
Next, we prove that the inner product of two $L^2$ functions remains unchanged when the functions are replaced by corresponding set functions. 
\begin{proposition}\label{prop:inner}
    Let us consider two measurable functions $f_{1}$ and $f_{2}$. Let $\tilde{f}_{1}$ and $\tilde{f}_{2}$ be the corresponding set functions respectively. Then $\langle f_{1}, f_{2} \rangle = \langle \tilde{f}_{1}, \tilde{f}_{2} \rangle $.
\end{proposition}
\begin{proof}
    First, consider the case where $f_{1}$ and $f_{2}$ are step functions on intervals, i.e., $f_{1}= \sum_{i=1}^{l_{1}} c_{1,i} \mathbb{I}_{(a_{i},b_{i}]}$ and $f_{2}= \sum_{j=1}^{l_{2}} c_{2,j} \mathbb{I}_{(a_{j},b_{j}]}$. One can prove that in this case, $\langle f_{1},f_{2} \rangle = \langle \tilde{f}_{1}, \tilde{f}_{2} \rangle $ by the arguments similar to Riemann integration. Now, we move on to general measurable functions. Let $f_{1}$ and $f_{2}$ be two functions in $L^{2}[0,1]$. For any given $\varepsilon >0$, we can find two step functions on intervals $f_{1,\varepsilon}$ and $f_{2,\varepsilon}$ such that 
    $\max\{||f_{1}-f_{1,\varepsilon}||_{2}, ||f_{2}-f_{2,\varepsilon}||_{2} \} < \varepsilon$.
    We also know that $\langle \tilde{f}_{1,\varepsilon}, \tilde{f}_{2,\varepsilon} \rangle = \langle f_{1,\varepsilon}, f_{2,\varepsilon}\rangle$ and we can find sequences of functions $f_{1, \frac{1}{k}}$ and $f_{2,\frac{1}{k}}$ such that 
    $\langle f_{1,\frac{1}{k}}, f_{2, \frac{1}{k}} \rangle \to \langle f_{1}, f_{2} \rangle $.

    \noindent
    We now prove a general bound for 
    \[
       \left|\sum_{j=1}^{l} \tilde{f}((a_{j-1},a_{j}]) \tilde{g}((a_{j-1},a_{j}]) (a_{j}-a_{j-1})\right|
    \]
    for two $L^2$ functions $f$ and $g$.
    \begin{equation}\label{eq:bound}
        \begin{split}
         &\left| \sum_{j=1}^{l} \tilde{f}((a_{j-1},a_{j}]) \tilde{g}((a_{j-1},a_{j}]) (a_{j}-a_{j-1})\right|\\
         &= \left| \sum_{j=1}^{l} \tilde{f}((a_{j-1},a_{j}]) \tilde{g}((a_{j-1},a_{j}]) \sqrt{(a_{j}-a_{j-1})} \sqrt{(a_{j}-a_{j-1})}\right|\\
         & \le \left| \sum_{j=1}^{l} \tilde{f}^2((a_{j-1},a_{j}])(a_{j}-a_{j-1}) \right|^{\frac{1}{2}}\times \left| \sum_{j=1}^{l} \tilde{g}^2((a_{j-1},a_{j}])(a_{j}-a_{j-1}) \right|^{\frac{1}{2}} \\
         &= \left| \sum_{j=1}^{l} \frac{\left(\int_{a_{j-1}}^{a_{j}} f(x) dx\right)^2}{(a_{j}-a_{j-1})}\right|^{\frac{1}{2}} \times \left| \sum_{j=1}^{l} \frac{\left(\int_{a_{j-1}}^{a_{j}} g(x) dx\right)^2}{(a_{j}-a_{j-1})}\right|^{\frac{1}{2}}\\
         &\le \left| \sum_{j=1}^{l} \int_{a_{j-1}}^{a_{j}} f^{2}(x)dx \right|^{\frac{1}{2}} \times \left| \sum_{j=1}^{l} \int_{a_{j-1}}^{a_{j}} g^{2}(x)dx \right|^{\frac{1}{2}}\\
         & = ||f||_{2}||g||_{2}.
        \end{split}
    \end{equation}
    For any partition $\mathcal{P}=\{ a_{0}=0< a_{1}< \ldots < a_{l}=1 \}$, we denote, 
    \[
    \langle \tilde{f}, \tilde{g} \rangle_{\mathcal{P}}:= \sum_{j=1}^{l} \tilde{f}((a_{j-1},a_{j}]) \tilde{g}((a_{j-1},a_{j}]) (a_{j}-a_{j-1}).
    \]
    Hence,
    \begin{equation}
         \langle \tilde{f}_{1}, \tilde{f}_{2} \rangle_{\mathcal{P}}=  \langle \tilde{f}_{1,\frac{1}{k}}, \tilde{f}_{2, \frac{1}{k}} \rangle_{\mathcal{P}} + \langle \tilde{f}_{1}-\tilde{f}_{1,\frac{1}{k}}, \tilde{f}_{2, \frac{1}{k}} \rangle_{\mathcal{P}} + \langle \tilde{f}_{1,\frac{1}{k}}, \tilde{f}_{2}-\tilde{f}_{2, \frac{1}{k}} \rangle_{\mathcal{P}} + \langle \tilde{f}_{1}-\tilde{f}_{1,\frac{1}{k}}, \tilde{f}_{2}-\tilde{f}_{2, \frac{1}{k}} \rangle_{\mathcal{P}}
    \end{equation}
    Hence, 
    \begin{equation} \label{eq:compare}
        \begin{split}
     &\langle {f}_{1,\frac{1}{k}}, {f}_{2, \frac{1}{k}} \rangle - \frac{1}{k}\left( ||\tilde{f}_{1, \frac{1}{k}}||_{2} +  ||\tilde{f}_{2, \frac{1}{k}}||_{2} + \frac{1}{k}\right)\\
     &\le \liminf_{||\mathcal{P}|| \to 0} \langle \tilde{f}_{1}, \tilde{f}_{2} \rangle_{\mathcal{P}} \le \limsup_{||\mathcal{P}|| \to 0} \langle \tilde{f}_{1}, \tilde{f}_{2} \rangle_{\mathcal{P}}\\
     &\le   \langle {f}_{1,\frac{1}{k}}, {f}_{2, \frac{1}{k}} \rangle + \frac{1}{k}\left( ||\tilde{f}_{1, \frac{1}{k}}||_{2} +  ||\tilde{f}_{2, \frac{1}{k}}||_{2} + \frac{1}{k}\right). 
        \end{split}
    \end{equation}
  \eqref{eq:compare} holds for any $k$. So, we take $k$ to $\infty$. Next, we observe that $ \langle f_{1,\frac{1}{k}} , f_{2,\frac{1}{k}}\rangle \to \langle f_{1}, f_{2} \rangle$, $||f_{1,\frac{1}{k}}||_{2} \to ||f_{1}||_{2}$ and $||f_{2, \frac{1}{k}}||_{2} \to ||f_{2}||_{2}$. Hence, the limit in \eqref{eq:compare} actually exists and is equal to $\langle f_{1}, f_{2} \rangle $.  
\end{proof}
\noindent 
\textbf{Properties of $\frac{dB(x)}{\sqrt{dx}}$:}
As the quantity $\frac{dB(x)}{\sqrt{dx}}$ is a quantity of fundamental interest to us, we would mention two of it's important properties. Firstly, it's orthogonality with any random/nonrandom measurable function and secondly, the fact that $\int_{0}^{1} \left( \frac{dB(x)}{\sqrt{dx}} \right)^{2} dx =1$ with high probability. 
\begin{lemma}
 $\sup_{f \in L^{2}[0,1] ~|~ ||f||_{2}=1}\left|\langle f, \frac{dB(x)}{\sqrt{dx}} \rangle\right|$ is well defined and is $0$ with probability $1$. Further, $\int_{0}^{1} \left( \frac{dB(x)}{\sqrt{dx}} \right)^{2} dx=1$ with probability $1$.
\end{lemma}
\begin{proof}
We at first prove the second part. We take a partition $\mathcal{P}=(0=a_{0}<a_{1}<\ldots< a_{l}=1)$ and write the Riemann type sum for the integral $\int_{0}^{1} \left( \frac{dB(x)}{\sqrt{dx}} \right)^{2} dx$. It is given by 
\begin{equation}\label{eq:dbx2}
\begin{split}
&\sum_{j=1}^{l} \left(\frac{B(a_{j})-B(a_{j-1})}{(a_{j}-a_{j-1})^{\frac{1}{2}}}\right)^{2}(a_{j}-a_{j-1})\\
&= \sum_{j=1}^{l} \left(B(a_{j}) - B(a_{j-1}) \right)^2.
\end{split}
\end{equation}
It is a well-known fact in Ito calculus that the second display in \eqref{eq:dbx2} goes to $1$ in probability.

\noindent 
Now, we move to the first part of the proof. Firstly, observe that for any partition $\mathcal{P}$,
\begin{equation}
\begin{split}
&\sum_{j=1}^{l} \left(f((a_{j-1},a_{j}])+g((a_{j-1},a_{j}])\right) \left( \frac{dB}{\sqrt{dx}} \right)\left((a_{j-1},a_{j}]\right)(a_{j}-a_{j-1})\\
&= \sum_{j=1}^{l}\left( f((a_{j-1},a_{j}]) \right)\left( \frac{dB}{\sqrt{dx}} \right)\left((a_{j-1},a_{j}]\right)(a_{j}-a_{j-1})+\\
&~~~~~~~ \sum_{j=1}^{l} \left( g((a_{j-1},a_{j}]) \right)\left( \frac{dB}{\sqrt{dx}} \right)\left((a_{j-1},a_{j}]\right)(a_{j}-a_{j-1}).
\end{split}
\end{equation}
Hence, provided the limits exist, $\langle f+g , \frac{dB}{\sqrt{dx}} \rangle = \langle f,  \frac{dB}{\sqrt{dx}}  \rangle + \langle g, \frac{dB}{\sqrt{dx}}  \rangle$. Secondly, we know by Cauchy-Schwarz inequality and \eqref{eq:bound}, 
\begin{equation}
\begin{split}
&\left|\sum_{j=1}^{l}\left( f((a_{j-1},a_{j}]) \right)\left( \frac{dB}{\sqrt{dx}} \right)\left((a_{j-1},a_{j}]\right)(a_{j}-a_{j-1})\right|^2 \\
& \le \left(\int_{0}^{1} f^{2}(x)dx \right)\times \left( \sum_{j=1}^{l} \left(B(a_{j})-B(a_{j-1})\right)^2 \right).
\end{split}
\end{equation}
Hence, when $||f_{n}|| \to 0 $, we have $\langle f_{n}, \frac{dB}{\sqrt{dx}} \rangle  \to 0$.
Now, we consider any Lebesgue measurable function $f(x)$ with $||f(x)||_{2}=1$. From the definition of measurability we know that for every $\varepsilon >0$, there is a function $g_{\varepsilon}(x)= \sum_{i=1}^{K} c_{i} \mathbb{I}_{(a_{i},b_{i}]}$ 
with $c_{i} \in \mathbb{Q}$ and $a_{i}< b_{i} \in \mathbb{Q}$ such that $|| f(x)- g_{\varepsilon}(x) ||^2 < \varepsilon$. Now, we prove $\sup_{c_{i},a_{i},b_{i} \in \mathbb{Q}}\left|\langle g_{\varepsilon}(x) , \frac{dB}{\sqrt{dx}}\rangle\right|=0$ almost everywhere.
To prove this, it is enough to consider $f$ which is equal to $c \in \mathbb{Q}$ on an interval $(a,b]$ such that $a,b \in \mathbb{Q}$ and $0$ otherwise.
Now, for any partition $\mathcal{P}$ of $(a,b]$, consider
\begin{equation}
\begin{split}
&\sum_{j=1}^{l} c \frac{B(a_{j})-B(a_{j-1})}{\left(a_{j}-a_{j-1}\right)^{\frac{1}{2}}}(a_{j}-a_{j-1})\\
&= c \times G(\mathcal{P})
\end{split}
\end{equation}
where $G(\mathcal{P}) \sim N(0, \sum_{j=1}^{l}(a_{j}-a_{j-1})^2)$. Observe that 
\[
\sum_{j=1}^{l} (a_{j}-a_{j-1})^2\le ||\mathcal{P}||(b-a) \to 0
\] 
as $||\mathcal{P}|| \to 0$. So, $\langle f, \frac{dB}{\sqrt{dx}} \rangle=0$ for any $f$ which takes value $c$ on $(a,b]$ and $0$ otherwise. Here $c,a,b \in \mathbb{Q}$. However, there are only countably many choices for $c,a,b$. Hence, we can say $\sup_{c,a,b \in \mathbb{Q}} \left|\langle f, \frac{dB}{\sqrt{dx}} \rangle\right|=0$ with probability $1$. This will imply $\sup_{c_{i},a_{i},b_{i} \in \mathbb{Q}}\left|\langle g_{\varepsilon}(x), \frac{dB(x)}{\sqrt{dx}} \rangle\right| =0$ almost everywhere. Hence, $\sup_{f \in L^{2}~|~ ||f||_{1}=1}\left|\langle f(x), \frac{dB(x)}{\sqrt{dx}} \rangle\right| =0$ almost everywhere for all $f$ measurable. 
\end{proof} 
\subsection{Defining the notion of convergence}
By hypothesis, given any $n$, we have a matrix of dimension $n \times n$. Hence, for every possible $n$, the objects we get are well defined. In particular, $K_{n,1}(f)$ is a well-defined vector on $\mathbb{R}^n$. However, the objects defined in Subsection \ref{subsec:delta} come only as a limit of well-defined objects as $n$ goes to infinity. In this section, we formally define the notion of convergence which will be useful to define the convergence of objects like $K_{n,1}(1)$ to $\frac{dB}{\sqrt{dx}}$.

\noindent 
To begin with, we have realized the Wigner matrix as a kernel operator on the space $L^{2}\{\frac{1}{n},\ldots, \frac{n}{n} \}$ with respect to the measure $\mu_{n}$ which can be viewed as an approximation of $\left([0,1], \lambda\right)$ where $\lambda$ is the Lebesgue measure on $[0,1]$. The vectors $f=(a_{1},\ldots, a_{n})$ in $\mathbb{R}^n$ we shall deal with will be realized as step functions on the interval $[0,1]$ where 
\begin{equation}
f(x)= a_{j} ~ \text{whenever } \frac{j}{n}\le x < \frac{j+1}{n}.
\end{equation}
We shall define the $L^2$ norm of $f$ with respect to $\mu_{n}$ to be 
\begin{equation}
||f||^2= \frac{1}{n} \sum_{i=1}^{n} a_{i}^2
\end{equation}
and for two functions $f$ and $g$ with $f=(a_{1},\ldots, a_{n})$ and $g=(b_{1},\ldots, b_{n})$ we define the inner product between $f$ and $g$ to be 
\[
\langle f,g \rangle = \frac{1}{n} \sum_{i=1}^{n} a_{i}b_{i}.
\]
Observe that due to the scaling $\frac{1}{n}$, the constant functions are of bounded norm. For example, one might consider the approximation of the indicator function on a fixed interval  $[a,b] \subseteq [0,1]$: 
\begin{equation}\label{eq:stepfunction}
f_{n}(x)= \left\{
\begin{array}{ll}
1 & \text{If } \frac{[an]}{n} \le x < \frac{[bn]}{n}\\
0 & \text{Otherwise}
\end{array}
\right.
\end{equation}
Then $||f_{n}||^2 = \left(1+O\left( \frac{1}{n} \right)\right)(b-a)$.

\noindent 
The main objective here is to incorporate random vectors like $f_{n}=(a_{1},\ldots, a_{n})$ such that $a_{i}$'s are i.i.d. standard Gaussian. We have discussed earlier that the limit of $f_{n}$ might not be a measurable function. However, $|| f_{n} ||^2=(1+o(1))$ almost surely. Next, we give a definition of convergence of $f_{n}$. 
\begin{definition}
For any step function $f_{n}$ which is constant on $\frac{j}{n}\le x < \frac{j+1}{n}$, we define 
\[
I_{f_{n},\frac{1}{2}}((a,b])= \sum_{j=[an]+1}^{[bn]} f_{n}\left( \frac{j}{n} \right)\left( \frac{1}{n}\right)^{\frac{1}{2}}. 
\]
\end{definition}
\begin{remark}
One might think $I_{f_{n},\frac{1}{2}}((a,b])$ to be a representative of the integral with respect to $\sqrt{dx}$ as we take $\left( \frac{1}{n} \right)^{\frac{1}{2}}$ in the definition. In fact, this interpretation is correct in the limit $n \to \infty$. However, as $f_{n}$'s are valid step functions, $\int_{0}^{1} f_{n}(x) (dx)^{\frac{1}{2}}$ is not well defined. This is because the integral $\int_{0}^{1} f_{n}(x) (dx)^{\frac{1}{2}}$ is defined when the mesh length $|| \mathcal{P} || \to 0$. However, for fixed $n$, $f_{n}$ is a step function where the limit is not defined when we take $||\mathcal{P}||$ much smaller than $\frac{1}{n}$. So, in order to define the limit, it is necessary that we take only meshes $\mathcal{P}_{n}$ to be $ 0< \frac{1}{n}< \ldots < \frac{n}{n}$ as $n \to \infty$. 
\end{remark}
\begin{definition}\label{def:con}
Let $\tilde{f}_{n}$ be a possibly random step-function which is constant on $\frac{j}{n}\le x < \frac{j+1}{n}$ and $g$ be a continuous function. Then, we say that $\tilde{f}_{n}$ is a consistent approximation of $\frac{dg(x)}{(dx)^{\frac{1}{2}}}$ if there is another step function $f_{n}$ constant on $\frac{j}{n}\le x < \frac{j+1}{n}$ such that the following conditions hold:
\[
\left( I_{f_{n},\frac{1}{2}}((a,b]) \right)_{a,b \in \mathbb{Q}\cap [0,1]} \stackrel{d}{\to} \left( g(b)-g(a) \right)_{a,b \in \mathbb{Q}\cap [0,1]}
\]
and 
\[
\sum_{j=1}^{n} \left( f_{n}\left( \frac{j}{n}\right) - \tilde{f}_{n} \left( \frac{j}{n} \right)\right)^{2} \left( \frac{1}{n} \right) \stackrel{p}{\to} 0.
\]
\end{definition}
In this paper, we shall require countably many $g_{i}$'s instead of a single $g$. The extension of Definition \ref{def:con} to the countable case is as follows:
\begin{definition}\label{def:conII}
Let $\{g_{i}\}_{i \in \mathbb{N}}$ be a countable sequence of possibly random continuous functions. As an example, one can think $g_{i}$ as a sequence of i.i.d. standard Brownian motions on $[0,1]$. For every $i$, let $\tilde{f}_{n,i}$ be a step function which is constant on $\frac{j}{n}\le x < \frac{j+1}{n}$. We say that $\left(\tilde{f}_{n,1},\ldots, \tilde{f}_{n,k},\ldots\right)$ is  jointly  consistent approximation of $\left(\frac{dg_{1}(x)}{(dx)^{\frac{1}{2}}},\ldots, \frac{dg_{k}(x)}{(dx)^{\frac{1}{2}}},\ldots,\right)$ if 
\[
\left( I_{f_{n,1},\frac{1}{2}}((a_{1},b_{1}]),\ldots, I_{f_{n,k},\frac{1}{2}}((a_{k},b_{k}]) \right)_{a_{1},b_{1},\ldots, a_{k},b_{k} \in \mathbb{Q}} \stackrel{d}{\to} \left( g_{1}(b_{1})-g_{1}(a_{1}),\ldots, g_{k}(b_{k})- g_{k}(a_{k})  \right)_{a_{1},b_{1},\ldots, a_{k},b_{k} \in \mathbb{Q}}
\]
and 
\[
\sum_{j=1}^{n} \left( f_{n,k}\left( \frac{j}{n}\right) - \tilde{f}_{n,k} \left( \frac{j}{n} \right)\right)^{2} \left( \frac{1}{n} \right) \stackrel{p}{\to} 0.
\]
for every fixed $k$.
\end{definition}
\begin{remark}\label{rem:boundedzero}
    In this case, one can also take examples following Remark \ref{rem:hausdorff} such that $||f_{n,k}||^{2} \to 0$ but $I_{f_{n,k},\frac{1}{2}}((a,b])$ is not well defined. We map these vectors to the $0$ vector. The kernel $K_{n}$ has operator norm bounded by $2$ almost surely. Hence, ignoring vectors with $L^{2}$ norm going to $0$ will not create any problem defining the limiting operator.
\end{remark}
Next, we state the following result which states that if one has the step function of the following form:
\begin{equation}\label{der:browsqrt}
f_{n}(x)= X_{n,i} ~ \text{whenever } \frac{i}{n}\le x < \frac{i+1}{n}
\end{equation}
with $X_{n,1},\ldots, X_{n,n}$ are i.i.d. with mean $0$ and variance $1$, then $f_{n}(x)$ is a consistent approximation of $\frac{dB}{\sqrt{dx}}$. This is a straightforward consequence of Donsker's theorem \cite{donsker1951invariance} which is stated below. 
\begin{theorem}
Let $X_{1},\ldots, X_{n},\ldots$ be a sequence of i.i.d. random variables with mean $0$ and variance $1$. Define the functions $W^{(n)}(t) \in C[0,1]$ as follows: If $t= \frac{i}{n}$ for some $i \in \{ 0,\ldots, n \}$, then $W^{(n)}(t)= \frac{\sum_{j=1}^{i}X_{j}}{\sqrt{n}}$. Otherwise, linearly interpolate between $W^{(n)}(\frac{i}{n})$ and $W^{(n)}(\frac{i+1}{n})$ to get the value of $W^{n}(t)$. Here $i$ is the unique integer such that $\frac{i}{n}\le t \le \frac{i+1}{n}$. Then $W^{(n)}(t)$ converges in distribution to a standard Brownian motion $B(t)$.
\end{theorem} 
\subsection{Results on the limit of $K_{n}$ }
The limiting operator can be defined in terms of countable infinitely many independent Brownian sheets. First, we define the Brownian sheets.
\begin{definition}(Brownian sheet)
A Brownian sheet $B(x,y)$ is a Gaussian process from $\mathbb{R}^2 \to \mathbb{R}$ where $\Cov[B(x_{1},y_{1}), B(x_{2},y_{2})]= \min\{ x_{1}, x_{2} \}\times \min\{y_{1},y_{2}\}$. We shall consider the restriction of the Brownian sheets on $[0,1] \times [0,1]$ only.
\end{definition}
In this paper, we shall require the following properties of the Brownian sheet repeatedly. 
\begin{theorem}
    Let $B(x,y)$ be a standard Brownian sheet. Then for any fixed $a<b \in \mathbb{R}$, the function $B(x,b)-B(x,a)$ as a function of $x$ is distributed as $(b-a)^{\frac{1}{2}}B(x)$ where $B(x)$ is a standard Brownian motion.
\end{theorem}
Using Definition \ref{def:dersetalpha},  we define the object $\frac{d(B(x,b)-B(x,a))}{\sqrt{dx}}$.
Recall that for the kernel $K_{n}\left( \cdot, \cdot \right)$ we have the constraint that $K_{n}\left(\frac{i}{n}, \frac{j}{n}\right)= K_{n}\left(\frac{j}{n}, \frac{i}{n}\right)$. In order to work with the constraint, we shall work with the symmetrized version of the Brownian sheet. In particular, we work with two Brownian sheets $B(x,y)$ and $B'(x,y)$ such that $B(x,y)= B'(y,x)$.

The following theorem proves a result about the realization of $K_{n}\left(\mathbb{I}_{(a,b]}\right)$
\begin{theorem}
Let $K_{n}(\cdot, \cdot)$ be as defined in \eqref{def:kernel}, $B(x,y)$ and $B'(x,y)$ be Brownian sheets with $B'(x,y)=B(y,x)$. Define $\Gamma(x,y)= \frac{1}{\sqrt{2}}\left( B(x,y) + B'(x,y) \right)$. Then for any interval $(a,b]$, $K_{n}(I_{(a,b]})$ is a consistent approximation (recalling Definition \ref{def:con}) of $\frac{d\left(\Gamma(x,b)-\Gamma(x,a)\right)}{\sqrt{dx}}$.
\end{theorem}
More generally, 
\begin{theorem}\label{thm:joinopconv}
Let $B_{m}(x,y)$ be a sequence of independent Brownian sheets and let $B_{m}'(x,y)=B_{m}(x,y)$. Define $C_{l}(x,y)= \sum_{m=1~|~ l-m ~ \text{even}}^{l}\frac{m+1}{l+1}\binom{l+1}{\frac{l-m}{2}}B_{m}(x,y)$ and $C_{l}'(x,y)= C_{l}(y,x)$ with $$\Gamma_{l}(x,y)= \frac{1}{\sqrt{2}} \left( C_{l}(x,y)+ C_{l}'(x,y) \right).$$ Also, let $P_{i}$ be an enumeration of the intervals $(a,b]$ with $0\le a<b \le 1$ and $a,b \in \mathbb{Q}$. Then for any fixed $l$, 
$K_{n}^{l} \left(\mathbb{I}_{P_{i}}\right)= \beta_{l} \mathbb{I}_{P_{i}} + \psi_{n,l,i}$ where $\psi_{n,l,i}$'s are jointly consistent approximations of $\left( \frac{d\left( \Gamma_{l}(x,b_{i}) -\Gamma_{l}(x,a_{i})\right)}{\sqrt{dx}}\right)_{i \in \mathbb{N}, l \in \mathbb{N}} $. Here, $\beta_{l}$ is the $l$ th moment of the semicircular law.  
\end{theorem}
This Theorem is perhaps the most important step in proving the convergence of the Wigner matrices. We shall provide a proof of this in the next section.
\subsection{Construction of the Hilbert space and the limiting operator}
We shall consider the elements of the Hilbert space on which the limiting operator is defined.
Instead of considering measurable functions, this Hilbert space contains set functions on half-open intervals.
We first give some elementary properties of the Hilbert space: 
\begin{enumerate}
\item Firstly, we incorporate any Lebesgue measurable and square-integrable function. Let $f$ be such a function. Then we write $f((a,b])= \frac{F(b)-F(a)}{b-a}$ where $F(b)-F(a)= \int_{a}^{b} f(x) dx$.
\item Next, we take a generic Brownian motion $B(x)$ and consider $\frac{dB(x)}{\sqrt{dx}}$. Here, for any $a<b$, we write $\frac{dB(x)}{\sqrt{dx}}((a,b])= \frac{B(b)-B(a)}{\sqrt{b-a}}$.
\item For any two set functions $f$ and $g$ and constants $c_{1}, c_{2} \in \mathbb{R}$, we define $(c_{1}f+c_{2}g)((a,b])= c_{1}f((a,b])+c_{2}g((a,b])$.
\item Finally for any two set functions $f$ and $g$, we define the inner product 
\begin{equation}\label{eq:innerproduct}
\langle f,g \rangle = \lim_{||\mathcal{P}|| \to 0} \sum_{i=1}^{l} f\left((a_{i-1},a_{i}]\right)g\left((a_{i-1},a_{i}]\right)(a_{i}-a_{i-1}).
\end{equation}
 Here, $\mathcal{P}= \{ 0=a_{0}<a_{1}<\ldots< a_{n}=1  \}$ and $||\mathcal{P}||= \sup_{i}\{ (a_{i}-a_{i-1}) \} $.
\end{enumerate}
\begin{remark}\label{rem:limitspace}
    We have seen in Remark \ref{rem:intequi} that several different set functions $G$ can be equivalent $\frac{dF}{\sqrt{dx}}$ yet $\int_{a}^{b} G(x)\sqrt{dx}$ might not be defined. As we are working with bounded operators, we are at liberty to consider only one $\mathbf{0}$ vector. This is the set function where $\mathbf{0}((a,b])=0$ for all $(a,b]$. Hence, we take the unique candidate in each equivalence class and $\frac{dB(x)}{\sqrt{dx}}((a,b])= \frac{B(b)-B(a)}{\sqrt{b-a}}$. In the vector space consider, we shall get $\int_{a}^{b} \frac{dB(x)}{\sqrt{dx}} \sqrt{dx}= B(b)-B(a)$. 
\end{remark}
\begin{remark}
The Hilbert space of interest will be defined \emph{explicitly} in terms of independent Brownian sheets and the concepts of derivatives introduced earlier in the paper. The objects in this Hilbert space will be random. In particular, for any $f$, $f((a,b])$ will be a random variable. Hence, the convergence in \eqref{eq:innerproduct} shall be in probability convergence. In addition, for the space with which we are working, the inner product will be well defined for any $f$ and $g$. 
\end{remark}
Now, we give an explicit definition of the Hilbert space and the limiting operator:
\begin{definition}\label{def:limspace}
Let $\Gamma_{l}(x,y)$ be as defined in Theorem \ref{thm:joinopconv}. Then the Hilbert space $$\mathcal{H} = \overline{\mathrm{Sp}}\left\{ \left(\mathbb{I}_{P_{i}}\right)_{i \in \mathbb{N}}, \left( \frac{d\left( \Gamma_{l}(x,b_{i}) -\Gamma_{l}(x,a_{i})\right)}{\sqrt{dx}}\right)_{i \in \mathbb{N}, l \in \mathbb{N}}   \right\}.$$ Here, $\overline{\mathrm{Sp}}$ denotes the closed span of the vectors. 
\end{definition}
\begin{definition}\label{def:limop}
Let the linear operator $\mathcal{W}: \mathcal{H} \to \mathcal{H}$ be defined in the following way:
\begin{equation}\label{eq:operatordefinition}
\begin{split}
\mathcal{W}(\mathbb{I}_{P_{i}}) &= \psi_{1,i}\\
\mathcal{W}\left(\psi_{k,i}\right) &= \beta_{k+1} \mathbb{I}_{P_{i}}+ \psi_{k+1,i} - \beta_{k}\psi_{1,i}.
\end{split}
\end{equation} 
Here,
$\psi_{k,i}:= \frac{d\left( \Gamma_{k}(x,b_{i}) -\Gamma_{k}(x,a_{i})\right)}{\sqrt{dx}}$.
\end{definition}
Now, we give the definition of the operator convergence. 
\begin{definition}\label{def:opconv}
    Let $K_{n}$ be an operator on $L^{2}\left\{ \frac{1}{n},\ldots, \frac{n}{n} \right\}$ with respect to the measure $\mu_{n}$ such that $||K_{n}||_{op}$ is bounded with high probability. Then, we say $K_{n} \rightsquigarrow \mathcal{W}$ if $K_{n}^{l}\left( \mathbb{I}_{P_{i}} \right)=\beta_{l}\mathbb{I}_{P_{i}}+ \psi_{n,l,i}+ \varepsilon_{n,l,i}$ where $||\varepsilon_{n,l,i}||^2 \stackrel{p}{\to} 0$, $\psi_{n,l,i}$'s are jointly consistent approximations of  $\left( \frac{d\left( \Gamma_{l}(x,b_{i}) -\Gamma_{l}(x,a_{i})\right)}{\sqrt{dx}}\right)_{i \in \mathbb{N}, l \in \mathbb{N}}$ and $\langle \psi_{n,l_{1},i_{1}}, \psi_{n,l_{2},i_{2}}  \rangle \stackrel{p}{\to} \left\langle \frac{d\left( \Gamma_{l_{1}}(x,b_{i_{1}}) -\Gamma_{l}(x,a_{i_{1}})\right)}{\sqrt{dx}}, \frac{d\left( \Gamma_{l_{2}}(x,b_{i_{2}}) -\Gamma_{l}(x,a_{i_{2}})\right)}{\sqrt{dx}} \right\rangle$ for every fixed $l_{1},l_{2}, i_{1}, i_{2}$.
\end{definition}
\section{Convergence of Wigner matrices}
In this section, we state the formal results about the convergence of Wigner matrices. 
\begin{theorem}\label{Thm:wigcon}
    Let $W_{n}$ be Wigner matrices of dimension $n \times n$ with $\mathrm{E}[|X_{i,j}|^{d}]\le (C_{1}d)^{C_{2}d}$ for all $d \in \mathbb{N}$ for some fixed $C_{1}$ and $C_{2}$. Consider the kernel $K_{n}$ as defined in Definition \ref{def:kernel}. Then $K_{n} \rightsquigarrow \mathcal{W}$ as $n \to \infty$.
\end{theorem}
We now state a Theorem for the spectral measure at any $L^{2}[0,1]$ measurable function.
\begin{theorem}\label{Thm:spect}
    The operator $\mathcal{W}$ is bounded and self-adjoint with probability $1$.
    Let $f \in L^{2}[0,1]$ be any fixed function and $\mu_{f}$ be the spectral measure of $\mathcal{W}$ at $f$. Then $\mu_{f}$ is almost surely the semicircular law.
\end{theorem}
\section{Proofs}
\subsection{Preliminaries}
The main ingredient of the proof of Theorem \ref{Thm:wigcon} and Theorem \ref{Thm:spect} is Theorem \ref{thm:joinopconv}. So, the main focus of this section will be to prove Theorem \ref{thm:joinopconv}. For this, we build some preliminaries. These definitions are taken from \cite{AGZ} and \cite{AZ05}.
\begin{definition}[$\mathcal{S}$ words]
Given a set $\mathcal{S}$, an $\mathcal{S}$ letter $s$ is simply an element of $\mathcal{S}$. An $\mathcal{S}$ word $w$ is a finite sequence of letters $s_1 \cdots s_k$, at least one letter long. An $\mathcal{S}$ word $w$ is \emph{closed} if its first and last letters are the same. In this paper, $\mathcal{S}=\{1,\ldots,n\}$.
\end{definition}
Two $\mathcal{S}$ words $w_1,w_2$ are called \emph{equivalent}, denoted $w_1\sim w_2$, if there is a bijection on $\mathcal{S}$ that maps one into the other. For any word $w = s_1 \cdots s_k$, we use $l(w) = k$ to denote its \emph{length}. We define the \emph{weight} $\wt(w)$ as the number of distinct elements of the set $\{s_1, \ldots, s_k \}$ and the \emph{support} of $w$, denoted by $\supp(w)$, as the set of letters appearing in $w$. With any word $w$, we may associate an undirected graph, with $\wt(w)$ vertices and at most $l(w) - 1$ edges, as follows.
\begin{definition}[Graph associated with a word]\label{def:graphword}
Given a word $w = s_1 \cdots s_k$, we let $G_w = (V_w, E_w)$ be the graph with set of vertices $V_w = \supp(w)$ and (undirected) edges $E_w = \{\{s_i, s_{i+1} \}, i = 1, \ldots, k - 1 \}.$
\end{definition}
The graph $G_w$ is connected since the word $w$ defines a path connecting all the
vertices of $G_w$, which further starts and terminates at the same vertex if the word
is \emph{closed}.  We note that equivalent words generate the same graphs $G_w$ (up to graph isomorphism) and the same passage-counts of the edges.
\begin{definition}[Sentences and corresponding graphs]\label{def:sentence}
A sentence $a = [w_i]_{i=1}^{m} = [[s_{i,j}]_{j = 1}^{l(w_i)}]_{i = 1}^{m}$ is an ordered collection of $m$ words of length $l(w_1), \ldots, l(w_m)$, respectively. We define $\supp(a) := \cup_{i = 1}^m \supp(w_i)$ and $\wt(a) := |\supp(a)|$. We set $G_a = (V_a, E_a)$ to be the graph with
\[
    V_a = \supp(a),\quad E_a = \left\{\{s_{i, j}, s_{i, j + 1}\} \mid j = 1, \ldots, l(w_i) - 1; \, i = 1,\ldots, m\right\}.
\]
\end{definition} 
Now, we write the expression for $K_{n}^{l}(\mathbb{I}_{P_{i}})$ for $P_{i}=(a_{i},b_{i}]$ for a given $l$.
\begin{equation}\label{eq:expressionfirst}
    \begin{split}
        K_{n}^{l}(\mathbb{I}_{P_{i}})\left(\frac{i}{n}\right)&= \left(\frac{\sqrt{n}}{n}\right)^{l} \sum_{j_{0},\ldots, j_{l} ~|~ j_{0}=i, j_{l} \in \{ [a_{i}n],\ldots, [b_{i},n] \}} x_{j_{0},j_{1}}\ldots x_{j_{l-1},j_{l}}.
    \end{split}
\end{equation}
Firstly, the index $i$ can be in $\{[an],\ldots, [bn]\}$ or not. If $i \in \{[an],\ldots, [bn]\}$, we write the r.h.s. of \eqref{eq:expressionfirst} in the following way:
\begin{equation}
\begin{split}
   & \left( \frac{1}{\sqrt{n}} \right)^{l}\sum_{j_{0},\ldots, j_{l} ~|~ j_{0}=i, j_{l} \in \{ [a_{i}n],\ldots, [b_{i},n] \}} x_{j_{0},j_{1}}\ldots x_{j_{l-1},j_{l}}\\
   &=   \underbrace{\left( \frac{1}{\sqrt{n}} \right)^{l}\sum_{j_{0},\ldots, j_{l} ~|~ j_{0}=i, j_{l}=i }   x_{j_{0},j_{1}}\ldots x_{j_{l-1},j_{l}}}_{V_{1}} +  \underbrace{\left( \frac{1}{\sqrt{n}} \right)^{l}\sum_{j_{0},\ldots, j_{l} ~|~ j_{0}=i, j_{l}\neq i }  x_{j_{0},j_{1}}\ldots x_{j_{l-1},j_{l}}}_{V_{2}} .
\end{split}
\end{equation}
On the other hand, if $i \notin \{[an] ,\ldots [bn] \}$, it is not possible to have $i_{0}=i_{l}$. The term $V_{1}$ actually has a non-vanishing mean when $l$ is even and a vanishing variance as $n \to \infty$. On the other hand, randomness comes from the term $V_{2}$ which has an asymptotically vanishing mean but a non-vanishing variance. We shall analyze these one by one. For proving the Gaussian convergence, we use method of moments and prove the joint moments satisfy the Wick's formula. 
Before going into the proof, we define some quantities which essentially capture the randomness. 
\begin{definition}\label{def:nonback}
    For any $l < n$, consider the corresponding random variables 
    \[
     \xi_{i,j,n}^{l} := \left( \frac{1}{\sqrt{n}} \right)^{l} \sum_{j_{0}=i,\ldots, j_{l}=j~|~ j_{0},\ldots, j_{l} ~\text{distinct}} x_{j_{0},j_{1}} \ldots x_{j_{l-1},j_{l}}
    \]
    whenever $i \neq j$ and $\xi_{i,i}^{l}=0$ whenever $i=j$. We shall call $\xi_{i,j,n}^{l}$ as the signed non-backtracking line of length $l$. Observe that $\xi_{i,j,n}^{l}=\xi_{j,i,n}^{l}$ for all $i \neq j$. Next, let us consider the following kernels on $L^{2}\{\frac{1}{n}, \ldots, \frac{n}{n} \}$ with respect to $\mu_{n}$ as defined in Section \ref{sec:wig}: 
    \[
     B_{n,l}\left(\frac{i}{n}, \frac{j}{n}\right):= n \xi_{i,j,n}^{l}.
    \]
\end{definition}
The entries $\xi_{i,j,n}^{l}$ will be of fundamental importance to prove Theorem \ref{Thm:wigcon}. We now state a lemma about the structures of the kernels $B_{n,l}$. 
\begin{lemma}\label{lem:nonback}
Let $B_{n,l}$'s be as defined in Definition \ref{def:nonback}. Then $B_{n,l}\left(\mathbb{I}_{P_{i}}\right)$'s are jointly consistent approximations of $\left( \frac{d(\Phi_{l}(x,b_{i})-\Phi_{l}(x,a_{i}))}{\sqrt{dx}} \right)_{i \in \mathbb{N},l\in \mathbb{N}}$. Here, $\Phi_l(x,y)= \frac{1}{\sqrt{2}}\left(B_{l}(x,y)+ B_{l}(y,x)\right)$ and $B_{l}(x,y)$ is a sequence of independent Browinan sheets.
\end{lemma}
We shall give a proof of this result later on. However, we shall essentially show that 
\[
\left(K_{n}^{l}(\mathbb{I}_{P_{i}})- \beta_{l}\mathbb{I}_{P_{i}}\right)_{i\in \mathbb{N}, l \in \mathbb{N}}
\]
has the same asymptotic distribution to 
\[
\left(\sum_{m=1~|~ l-m ~ \text{even}}^{l} \frac{m+1}{l+1}\binom{l+1}{\frac{l-m}{2}} B_{n,l}\left(\mathbb{I}_{P_{i}}\right)\right)_{l\in \mathbb{N}, m \in \mathbb{N}}
\]
as $n$ goes to infinity.
We shall now give a proof of Lemma \ref{lem:nonback}. For this, we require the statement of Wick's theorem and a combinatorial lemma. 
\begin{theorem}(Wick's Theorem)\label{Thm:wick}
    Suppose, we have a sequence of $l$ dimensional random vector $\left(Y_{n,1},\ldots, Y_{n,l}\right)$ and there exists an $l \times l$ positive semi-definite matrix $\Sigma$ such that for any $m$, 
    \[
      \mathrm{E}\left[ \prod_{i=1}^{m} X_{n,i} \right] \to \left\{
         \begin{array}{cc}
             0 & \text{whenever $m$ is odd} \\
             \sum_{\eta \in \mathcal{P}_{m}} \prod_{j=1}^{\frac{m}{2}} \Sigma\left(\eta(j,1),\eta(j,2)\right) & \text{whenever $m$ is even} 
         \end{array}
      \right.
    \]
    for any $X_{n,1},\ldots, X_{n,m} \in \{ Y_{n,1},\ldots, Y_{n,l} \}$. Here, for even $m$, $\mathcal{P}_{m}$ denotes all possible partitions of $\{ 1,\ldots,m \}$ such that each block has two elements and for a partition $\eta \in \mathcal{P}_{m}$, $\eta(j,1),\eta(j,2)$ denote the first and second entry of the $j$ th block of $\eta$. Then $(Y_{n,1},\ldots, Y_{n,l}) \stackrel{d}{\to} N_{l}\left( 0, \Sigma \right)$. We make a remark that there can be repetitions among the random variables $X_{n,1},\ldots, X_{n,m}$. In particular, taking $X_{n,1}=\ldots= X_{n,m}=Y_{n,1}$ will give the distributional convergence of $Y_{n,1}$ to $N(0, \Sigma(1,1))$. 
\end{theorem}
Next, we state a Lemma which is important for proving the asymptotic normality.
\begin{lemma}\label{lem:clt}
    Suppose, we have a sentence $a$ consisting of $m$ words $w_{1},\ldots,w_{m}$ such that none of the words $w_{i}$ are closed, all the edges in the graph $G_{a}$ are traversed at least twice and for all $i$, there is $j\neq i$ such that $w_{i}$ and $w_{j}$ shares an edge. Then $ \mathrm{wt}(a) \le \sum_{i=1}^{k} \frac{l(w_{i})}{2} $ and the equality occurs if only if all the following conditions are satisfied: 
    \begin{enumerate}[(a)]
        \item All the edges in the graph $G_{a}$ are traversed exactly twice.
        \item $G_{a}$ is a collection of $\frac{m}{2}$ disjoint forests. In particular, $m$ has to be even.
        \item For every $i$, there is a unique $j$ such that $w_{i}$ and $w_{j}$ shares an edge.
        \item If $w_{i}$ and $w_{j}$ share an edge, then there will be a unique path of some length in $G_{w_{i}}$ which is common with $G_{w_{j}}$. These edges will be traversed exactly once. All other edges will be traversed exactly twice.
    \end{enumerate}
\end{lemma}
\begin{proof}
    We know that for every $i$, there is $j\neq i$ such that $w_{i}$ and $w_{j}$ share an edge. Hence, the graph $G_{a}$ can be divided into several connected components. Now, the number of connected components will be less than equal to $\frac{m}{2}$ and the equality happens if and only if every connected component has two words. At first, we consider the case where one connected component has three or more words. We call this connected component $\mathcal{C}$. Let $E$ be the edge set of this connected component. Since, every edge is traversed at least twice, we have $|E|\le \sum_{i \in \mathcal{C}} \frac{l(w_{i})-1}{2}$. This follows from the following reason: Consider any word $w=(i_{0},\ldots, i_{k})$, then the corresponding random variable $x_{i_{0},i_{1}}\ldots x_{i_{k-1},i_{k}}$ has $k$ terms, while $l(w)=k+1$. Now, the graph corresponding to $\mathcal{C}$ has vertices $|V|\le |E|+1$. Hence $|V|\le \sum_{i \in \mathcal{C}} \frac{l(w_{i})-1}{2} + 1 \le \sum_{i \in \mathcal{C}} \frac{l(w_{i})}{2}$. The equality happens if and only if the graph $(V,E)$ of $\mathcal{C}$ is a tree, every edge is traversed exactly twice and the number of words in $\mathcal{C}$ is exactly two. This argument proves $(a)-(c)$ of the lemma. For part $(d)$, we refer to Lemma 4.4 in \cite{AZ05}. According to this lemma, only edges appearing in the geodesic with initial and terminal vertices coinciding with those of $w_{i}$ can be traversed odd number of times. As all the edges have been traversed at most twice, these edges are traversed exactly once. Since each edge in the sentence $[w_{i},w_{j}]$ is traversed exactly twice, $G_{w_{i}}$ overlaps with $G_{w_{j}}$ only on this geodesic. This completes the proof. Actually, one can show that when at least one of the four conditions ($(a)-(d)$) is violated, $\mathrm{wt}(a)\le \sum_{i=1}^{m} \frac{l(w_{i})}{2}- \frac{1}{2}$.
\end{proof}
\begin{proof}[Proof of Lemma \ref{lem:nonback}]
    We shall only consider two distinct values $l_{1}\neq l_{2}$ and two intervals $P_{1}=(a_{1},b_{1}],P_{2}=(a_{2},b_{2}]$. We shall prove that $\left(B_{n,l_{j}}(\mathbb{I}_{P_{j}})\right)_{j=1,2}$'s are jointly consistent approximations of $\left( \frac{d(\Phi_{l_{j}}(x,b_{j})- \Phi_{l_{j}}(x,a_{j}))}{\sqrt{dx}} \right)_{j\in \{ 1,2\}}$. 
    Let $Q_{i}$ be an enumeration of half-open intervals. Then we need to prove 
    \begin{equation}
        \left( X_{Q_{i},l_{1},P_{1},n},X_{Q_{i},l_{2},P_{2},n} \right) \stackrel{d}{\to}\left(  \Phi_{l_{1},P_{1}}(Q_{i}), \Phi_{l_{2},P_{2}}(Q_{i})\right).
    \end{equation}
    Here, for any $Q_{i}=(c_{i},d_{i}]$, $\Phi_{l_{j},P_j}(Q_{i})= \Phi_{l_{j}}(d_{i},b_{j})- \Phi_{l_{j}}(c_{i},b_{j})- \Phi_{l_{j}}(d_{i},a_{j})+ \Phi_{l_{j}}(c_{i},a_{j})$ for $j \in \{ 1,2 \}$ and 
    \begin{equation*}
       \begin{split}
        X_{Q_{i},l_{j},P_{j},n} &= \left( \frac{1}{\sqrt{n}} \right)^{l_{j}+1} \sum_{i_{0} \in \{[c_{i}n],\ldots, [d_{i}n]\},i_{1}\in [n],\ldots, i_{l_{j}} \in \{ [a_{j}n],\ldots, [b_{j}n] \}~|~ i_{0},\ldots,i_{l_{j}} ~ \text{are distinct}} x_{i_{0},i_{1}}\ldots x_{i_{l_{j-1}},i_{l_{j}}}\\
        &= \frac{1}{\sqrt{n}} \sum_{ i_{0} \in \{ [c_{i}n],\ldots, [d_{i},n] \}} B_{n,l_{j}}(\mathbb{I}_{P_{j}})\left(  \frac{i_{0}}{n}\right).
        \end{split}
    \end{equation*}
The case for multiple values of $l_{j}$'s is essentially same. This proof is broken into two parts: Firstly, we prove the joint asymptotic normality. Secondly, we prove that covariance matrix is same as covariance of the limiting quantity.\\
\textbf{Proof of joint asymptotic normality:} This proof follows from applications of Theorem \ref{Thm:wick} and Lemma \ref{lem:clt}. We choose finitely many intervals $(c_{1},d_{1}],\ldots (c_{k},d_{k}]$ and call them $Q_{1},\ldots, Q_{k}$. We prove 
$\left(X_{Q_{i},l_{1},P_{1},n},X_{Q_{i},l_{2},P_{2},n}\right)_{1 \le i \le k}$ jointly converges in distribution to a $2k$ dimensional Gaussian random variable with some variance covariance matrix for every $k$. Observe that $\mathrm{E}[X_{Q_{i},l_{j},P_{j},n}]=0$ for every $i$ and $j$. So, we consider the monomial 
$\Psi_{n,1}\ldots \Psi_{n,m}$ for any fixed $m$ where $\Psi_{n,s} \in \{ X_{Q_{i},l_{1},P_{1},n},X_{Q_{i},l_{2},P_{2},n} ~|~ 1 \le i \le k\} $ for every $s$. We shall prove Theorem \ref{Thm:wick} for $\mathrm{E}\left[ \Psi_{n,1}\ldots \Psi_{n,m} \right]$. Let $\Gamma_{s}$ denote the collection of words corresponding to the random variable $\Psi_{n,s}$. Observe that these words have a common length. We denote this by $l(\Gamma_{s})$. Hence, 
\[
\Psi_{n,s}= \left( \frac{1}{\sqrt{n}} \right)^{l(\Gamma_{s})} \sum_{w=(i_{0},\ldots,i_{l(\Gamma_{s})-1}) \in \Gamma_{s} } x_{i_{0},i_{1}}\ldots x_{i_{l(\Gamma_{s})-2},i_{l(\Gamma_{s})-2}}.
\]
For any word $w=(i_{0},\ldots,i_{l(w)-1})$, we denote 
$x_{w}:= x_{i_{0},i_{1}}\ldots x_{i_{l(w)-2},i_{l(w)-1}}$.
Hence,
\[
\Psi_{n,s}= \left( \frac{1}{\sqrt{n}} \right)^{l(\Gamma_{s})} \sum_{w \in \Gamma_{s}} x_{w}.
\]
So, 
\[
\mathrm{E}[\Psi_{n,1}\ldots \Psi_{n,m}]= \left( \frac{1}{\sqrt{n}} \right)^{\sum_{s=1}^{m}l(\Gamma_{s})} \sum_{w_{1},\ldots, w_{m}} \mathrm{E}\left[x_{w_{1}}\ldots x_{w_{m}}\right].
\]
For a generic choice of $w_{1},\ldots, w_{m}$, we consider the sentence $a=[w_{1},\ldots, w_{m}]$. Observe that $\mathrm{E}[x_{w_{1}}\ldots x_{w_{m}}] \neq 0$ only if all the edges in $G_{a}$ are traversed at least twice. As $l_{1},l_{2}$ and $m$ are fixed, for any choice of $(w_{1},\ldots,w_{m})$, there exists a uniform number $C_{l_{1},l_{2},m}$ such that 
\[
\mathrm{E}\left[ \left| x_{w_{1}}\ldots x_{w_{m}} \right| \right] \le C_{l_{1},l_{2},m}.
\]
Let $\mathcal{K}_{\ge 2}$ denote the class of all sentences $a=[w_{1},\ldots w_{m}]$ such that all the edges in $G_{a}$ are traversed at least twice and $\mathcal{K}_{2}$ where all the edges are traversed exactly twice. From Lemma \ref{lem:clt}, we get that any sentence $a \in \mathcal{K}_{\ge 2} \backslash \mathcal{K}_{2}$ will have $\mathrm{wt}(a)\le \sum_{s=1}^{m} l(\Gamma_{s}) -\frac{1}{2}$. On the other hand, the number of equivalence classes in of $a$ in $\mathcal{K}_{\ge 2}\backslash \mathcal{K}_{2}$ is again bounded by a universal constant $C'_{l_{1},l_{2},m}$. Hence,
\[
\left( \frac{1}{\sqrt{n}} \right)^{\sum_{s=1}^{m}l(\Gamma_{s})}\sum_{a \in \mathcal{K}_{\ge 2}\backslash \mathcal{K}_{2}} \mathrm{E}\left[\left| x_{w_{1}}\ldots x_{w_{m}} \right|\right] \le \left( \frac{1}{\sqrt{n}} \right) C_{l_{1},l_{2},m} C'_{l_{1},l_{2},m} \to 0.
\]
Hence, it is enough to work with $\mathcal{K}_{2}$ only. Now, by Lemma \ref{lem:clt}, we reduce our focus further to the class of sentences where conditions $(a)-(d)$ are satisfied. We call such sentences CLT sentences. This automatically gives rise to a partition $\mathcal{P}$ of the set $\{ 1,\ldots, m \}$ where each block contains exactly two elements. We put $s_{1}$ and $s_{2}$ in the same block if $w_{s_{1}}$ and $w_{s_{2}}$ have at least one edge common. In such cases $\mathrm{E}[x_{w_{s_{1}}}x_{w_{2}}]=1.$
Further, as the random variables corresponding to two different blocks are independent, we write 
\begin{equation}\label{eq:wicks}
    \begin{split}
        &\mathrm{E}\left[ \Psi_{n,1},\ldots, \Psi_{n,m} \right]\\
        &= \sum_{\eta \in \mathcal{P}_{m}} \left( \frac{1}{\sqrt{n}} \right)^{\sum_{s=1}^{m}l(\Gamma_{s})} \#\{a ~:~ \text{the partition corresponding to $a$ is $\eta$} \}.
    \end{split}
\end{equation}
We now consider a partition $\eta$ and all CLT sentences corresponding to the partition $\eta$.
It requires the words inside the same blocks are perfectly matched and the words in the different blocks have no edge in common. As $l_{1},l_{2}$ and $m$ are fixed, once we fix a word inside a block, all the vertices in the other blocks will have $n(1+o(1))$ many choices. Hence, the counts inside one block doesn't depend on the others up to a $(1+o(1))$ factor. On the other hand, the graph corresponding to each word is a straight line, hence a tree. So, the total power of $n$ is given by $n^{\frac{\sum_{s=1}^{m} l(\Gamma_{s})}{2}}$. Hence, we can reduce the count in \eqref{eq:wicks} further to 
\[
\sum_{\eta \in \mathcal{P}_{m}} (1+o(1)) \prod_{s=1}^{\frac{m}{2}} \frac{\#\{(w_{\eta(s,1)},w_{\eta(s,2)})~:~ w_{\eta(s,1)}= w_{\eta(s,2)}\}}{n^{l(\Gamma_{\eta(s,1)})}}.
\]
This almost proves the CLT provided we prove that $\lim_{n \to \infty} \frac{\#\{(w_{\eta(s,1)},w_{\eta(s,2)})~:~ w_{\eta(s,1)}= w_{\eta(s,2)}\}}{n^{l(\Gamma_{\eta(s,1)})}}$ exists as $n \to \infty$. This we shall prove in the next step.
\\
\textbf{Calculation of limiting variance:} 
First of all, observe that in $\frac{\#\{(w_{\eta(s,1)},w_{\eta(s,2)})~:~ w_{\eta(s,1)}= w_{\eta(s,2)}\}}{n^{l(\Gamma_{\eta(s,1)})}}$, $l(\Gamma_{\eta(s,1)})= l(\Gamma_{\eta(s,2)})$. Otherwise, the count is $0$. Let $w_{\eta(s,1)}= (i_{0},i_{1},\ldots, i_{l(\Gamma_{\eta(s,1)})-1})$ be a typical word. Observe that the choices for each $i_{j}$ for $1\le j \le l(\Gamma_{\eta(s,1)})-2$ are $(1+o(1))n$. Hence, the total number of choices for these vertices is $(1+o(1))n^{l(\Gamma_{\eta(s,1)})-2}$. So, it remains to prove that
\[
\lim_{n \to \infty} \frac{\text{no of choices for $i_{0}$ and $i_{l(\Gamma_{\eta(s,1)})-1}$}}{n^2}
\]
exists.

To calculate the variance, we use a trick. We consider two independent kernels $T_{n,l_1}$ and $T_{n,l_2}$ such that $T_{n,l_{s}}\left( \frac{i}{n}, \frac{j}{n} \right)= \frac{\sqrt{n}}{\sqrt{2}} \left(\varphi_{i,j,l_s} + \varphi_{j,i,l_s}\right)$ with $\varphi_{i,j,l_s} \sim_{i.i.d.} N(0,1)$ for $s \in \{ 1,2 \}$. We write 
$T_{n,l_{s}}^{(1)}(\frac{i}{n},\frac{j}{n})= \sqrt{n} \varphi_{i,j,l_{s}}$ and $T_{n,l_{s}}^{(2)}(\frac{i}{n},\frac{j}{n})= \sqrt{n} \varphi_{j,i,l_{s}}$. It can be easily shown that for intervals $P_{1},P_{2}$ and $Q_{1},\ldots, Q_{k}$ 
\begin{equation}
\begin{split}
&\left(\frac{1}{n} \sum_{i_{0}\in \{[nc_{q}],\ldots,[nd_{q}]\}, i_{1} \in \{[na_{j}],[nb_{j}]\}} \varphi_{i_{0},i_{1},l_{j}}, \frac{1}{n} \sum_{i_{0}\in \{[nc_{q}],\ldots,[nd_{q}]\}, i_{1} \in \{[na_{j}],[nb_{j}]\}} \varphi_{i_{1},i_{0},l_{j}}\right)_{1\le q \le k, 1 \le j \le 2}\stackrel{d}{\to}\\
&~~~~~~~~~~~~~\left( B_{l_{j}}(d_{q},b_{j})- B_{l_{j}}(c_{q},b_{j})- B_{l_{j}}(d_{q},a_{j}) + B_{l_{j}}(c_{q},a_{j}), \right.\\
& ~~~~~~~~~~~~~~~~~~~\left. B_{l_{j}}(b_{j},d_{q})- B_{l_{j}}(b_{j},c_{q})- B_{l_{j}}(a_{j},d_{q}) + B_{l_{j}}(a_{j},c_{q}) \right)_{1\le q \le k, 1 \le j \le 2}.
\end{split}
\end{equation}
Hence,
\begin{equation}\label{eq:Gaucon}
\begin{split}
&\left(\frac{1}{n\sqrt{2}} \sum_{i_{0}\in \{[nc_{q}],\ldots,[nd_{q}]\}, i_{1} \in \{[na_{j}],[nb_{j}]\}} \varphi_{i_{0},i_{1},l_{j}}+ \frac{1}{n\sqrt{2}} \sum_{i_{0}\in \{[nc_{q}],\ldots,[nd_{q}]\},i_{1} \in \{[na_{j}],[nb_{j}]\}} \varphi_{i_{1},i_{0},l_{j}}\right)_{1\le q \le k, 1 \le j \le 2}\stackrel{d}{\to}\\
&~~~~~~~~~~~~~\frac{1}{\sqrt{2}}\left( B_{l_{j}}(d_{q},b_{j})- B_{l_{j}}(c_{q},b_{j})- B_{l_{j}}(d_{q},a_{j}) + B_{l_{j}}(c_{q},a_{j})+ \right.\\
& ~~~~~~~~~~~~~~~~~~~\left. B_{l_{j}}(b_{j},d_{q})- B_{l_{j}}(b_{j},c_{q})- B_{l_{j}}(a_{j},d_{q}) + B_{l_{j}}(a_{j},c_{q}) \right)_{1\le q \le k, 1 \le j \le 2}.
\end{split}
\end{equation}
Also, the corresponding covariances converge. Now, 
\begin{equation}\label{eq:varlim}
    \begin{split}
       &\frac{\text{no of choices for $i_{0}$ and $i_{l(\Gamma_{\eta(s,1)})-1}$}}{n^2}\\
       &= \mathrm{Cov}\left( \left(\frac{1}{n\sqrt{2}} \sum_{i_{0}\in \{[nc_{1}],\ldots,[nd_{1}]\}, i_{1} \in \{[na_{1}],[nb_{1}]\}} \varphi_{i_{0},i_{1},l(\Gamma_{\eta(s,1)})}+ \frac{1}{n\sqrt{2}} \sum_{i_{0}\in \{[nc_{1}],\ldots,[nd_{1}]\},i_{1} \in \{[na_{1}],[nb_{1}]\}} \varphi_{i_{1},i_{0},l(\Gamma_{\eta(s,1)})},\right.\right.\\
       & ~~~~~~~~~~~~~~~~~~~~~~~~~~~~~~~~~~~~~~~~~~~~~~~\left. \left. \frac{1}{n\sqrt{2}} \sum_{i_{0}\in \{[nc_{2}],\ldots,[nd_{2}]\}, i_{1} \in \{[na_{2}],[nb_{2}]\}} \varphi_{i_{0},i_{1},l(\Gamma_{\eta(s,1)})}+ \frac{1}{n\sqrt{2}} \sum_{i_{0}\in \{[nc_{2}],\ldots,[nd_{2}]\},i_{1} \in \{[na_{2}],[nb_{2}]\}} \varphi_{i_{1},i_{0},l(\Gamma_{\eta(s,1)})} \right) \right)
    \end{split}
\end{equation}
where $(c_{1},d_{1}],(c_{2},d_{2}],(a_{1},b_{1}],(a_{2},b_{2}]$ are the intervals corresponding to $\eta(s,1)$ and $\eta(s,2)$. By using \eqref{eq:Gaucon}, we get that the limit in the L.H.S. of \eqref{eq:varlim} exists. This concludes the proof.
\end{proof}
\begin{proof}[Proof of Theorem \ref{thm:joinopconv}:]
Here, we consider a typical interval $P_{i}=(a,b]$. We write $K_{n}^{l}(\mathbb{I}_{P_{i}})$ in the following way:
\begin{equation}
    \begin{split}
        & K_{n}^{l}(\mathbb{I}_{P_{i}})\left(  \frac{i}{n}\right)\\
        &= \underbrace{\left( \frac{1}{\sqrt{n}} \right)^{l} \sum_{j_{0}=i, \ldots , j_{l}=i } x_{j_{0},j_{1}} \ldots x_{j_{l-1},j_{l}}}_{V_{1}(\frac{i}{n})\text{ say}} + \underbrace{\left( \frac{1}{\sqrt{n}} \right)^{l} 
        \sum_{j_{0}=i,\ldots, j_{l}\neq i} \left(x_{j_{0},j_{1}}\ldots x_{j_{l-1},j_{l}} -\mathrm{E}\left[ x_{j_{0},j_{1}}\ldots x_{j_{l-1},j_{l}}\right]\right)}_{V_{2}(\frac{i}{n}) \text{ say}}\\
        &~~~~~~ + \underbrace{\left( \frac{1}{\sqrt{n}} \right)^{l} 
        \sum_{j_{0}=i,\ldots, j_{l}\neq i} \mathrm{E}\left[ x_{j_{0},j_{1}}\ldots x_{j_{l-1},j_{l}} \right]}_{V_{3}(\frac{i}{n}) \text{ say}}.
    \end{split}
\end{equation}
We shall prove that $||V_{1}- \beta_{l} I_{P_{i}}||^{2} \stackrel{p}{\to} 0$, $||V_{3}||^{2} \to 0$ and 
\begin{equation}\label{eq:traceapproximationnb}
\left( \frac{1}{\sqrt{n}} \right)\sum_{i\in\{ [cn],\ldots,[dn\}]} V_{2}\left( \frac{i}{n}\right)- \sum_{m=1~|~ l-m ~ \text{even}}^{l} \frac{m+1}{l+1}\binom{l+1}{\frac{l-m}{2}}\underbrace{\left[\left( \frac{1}{\sqrt{n}} \right)\sum_{i \in  \{ [cn],\ldots, [dn]\}, j \in \{ [an],\ldots, [bn]\}  } \xi^{l}_{i,j,n}\right]}_{T_{2,m} (\text{ say})} \stackrel{p}{\to} 0.
\end{equation}

    \noindent 
    We prove this by comparing the variances and covariances.

    \noindent 
    First, we prove that $||V_{1}-\beta_{l}I_{P_{i}}||^{2} \stackrel{p}{\to} 0.$
    Firstly, when $i \notin \{ [an],\ldots,[bn] \}$, $V_{1}\left( \frac{i}{n} \right)=0$. So, we can consider the case $i \in \{ [an],\ldots, [bn] \}$. Here,
    $V_{1}\left( \frac{i}{n} \right)$ contains the closed words. Hence, it comes under the proof of \cite{AZ05}. If $l$ is even, then there are Dyck paths of length $l$ where all the edges will be repeated exactly twice. There are $\beta_{l}$ many Dyck paths and fixing a Dyck path and the initial point $i$, we shall have $n^{\frac{l}{2}}(1+o(1))$ many words corresponding to these Dyck paths. So, these paths will have mean $\beta_{l} \left( \frac{n}{n} \right)^{\frac{l}{2}}(1+o(1))=(1+o(1))\beta_{l}$. Now, following the proof of \cite{AZ05}, whenever $i \in \{ [an],\ldots,[bn] \}$, $\mathrm{E}\left[ V_{1}\left( \frac{i}{n} \right)-\beta_{l} \right]^{2}=O\left(\frac{1}{n}\right).$ As a consequence, $||V_{1}- \beta_{l}\mathbb{I}_{P_{i}}||^{2}= \frac{1}{n}\sum_{i=1}^{n}|V_{1}\left( \frac{i}{n} \right)- \beta_{l}\mathbb{I}_{i \in \{  [an],\ldots,[bn]\}}|^2 \stackrel{p}{\to} 0.$ 

    \noindent 
    Now we consider the term $V_{3}$. Recall that $V_{3}\left( \frac{i}{n} \right)= \left( \frac{1}{\sqrt{n}} \right)^{l} \sum_{j_{0}=i,\ldots, j_{l}\neq i}\mathrm{E}\left[ x_{j_{0},j_{1}}\ldots x_{j_{l-1},j_{l}} \right]$. Any term in the aforesaid sum has nonzero mean only if all the random variables in $x_{j_{0},j_{1}}\ldots x_{j_{l-1},j_{l}}$ are traversed more than twice. Let us fix a word $w$ in the aforesaid sum. We know that $|E_{w}| \le \frac{l}{2}$ as all the edges have been traversed at least twice. However, we have the constraint that $i_{0} \neq i_{l}$. If all the edges were traversed exactly twice, then the corresponding multi-graph would have been Eularian. Now, any exploration on an Eularian graph covering all the edges has to be closed. So, there is at least one edge in $E_{w}$ that is traversed at least thrice. So, $2(|E_{w}|-1)+3 \le l \Rightarrow |E_{w}|\le \frac{l-1}{2}.$ So, $|V_{w}|\le \frac{l-1}{2}+1$. However, we fixed the initial vertex $i$. So, total number of free choices of the vertices are bounded by $\frac{l-1}{2}$. Hence $|V_{3}\left( \frac{i}{n} \right)|\le \frac{n^{\frac{l-1}{2}}}{n^{\frac{l}{2}}}C=\frac{C}{\sqrt{n}}$. Here, $C$ is a universal constant. So,  $||V_{3}||^{2}= \frac{1}{n}\sum_{i=1}^{n} V_{3}\left( \frac{i}{n} \right)^2 \to 0.$ 

\noindent 
 We now calculate 
\[
\mathrm{Var}\left[\left( \frac{1}{\sqrt{n}} \right) \sum_{i \in \{[cn],\ldots,[dn]\} } V_{2}\left( \frac{i}{n} \right)\right].
\]
To make our notation simpler, for any word $w=(i_{0},\ldots,i_{l})$, we call $x_{w}= x_{i_{0},i_{1}}\ldots x_{i_{l-1},i_{l}}$. So 
\begin{equation}
    \begin{split}
       &\mathrm{Var}\left[ \left( \frac{1}{\sqrt{n}} \right) \sum_{i \in [cn],\ldots, [dn]  } V_{2}\left( \frac{i}{n} \right)\right]\\
       &= \left(\frac{1}{n}\right)^{l+1} \sum_{w_{1}, w_{2} ~|~ \text{ satisfying the constraints}} \mathrm{E}\left[ (x_{w_{1}}-\mathrm{E}(x_{w_{1}}))(x_{w_{2}}-\mathrm{E}(x_{w_{2}})) \right].
    \end{split}
\end{equation}
We know, $l(w_{1})=l(w_{2})=l+1$, $\left|\mathrm{E}\left[ (x_{w_{1}}-\mathrm{E}(x_{w_{1}}))(x_{w_{2}}-\mathrm{E}(x_{w_{2}})) \right]\right|\le C'$ for some universal $C'$ and in order to have $\mathrm{E}\left[ (x_{w_{1}}-\mathrm{E}(x_{w_{1}}))(x_{w_{2}}-\mathrm{E}(x_{w_{2}})) \right] \neq 0$, $G_{w_{1}}$ and $G_{w_{2}}$ needs to share an edge. Using Lemma \ref{lem:clt}, we get that $wt([w_{1},w_{2}])\le l+1$. We only consider the case when the equality happens. From Lemma \ref{lem:clt} $(a)-(d)$, we get that all the edges in the traversed exactly twice, $G_{a}$ is a tree and there is a unique path in $G_{w_{1}}$ which is common with $G_{w_{2}}$. We call this path to be the stem of $w_{1}$ and $w_{2}$. Let the length of the stem be $m$. As all the edges of $G_a$ is traversed exactly twice, the edges of the stem is traversed exactly once. Further, the set of starting point and the endpoint of the stem has to be same as the set of starting point and the endpoint of the words $w_{1}$ and $w_{2}$. Fixing the stem, $G_{w_{1}}$ is a collection of trees, corresponding to a Dyck path from every point of the stem. The total possible choices for the Dyck paths are given by the coefficient of $x^{\frac{l-m}{2}}$ in $C^{m+1}(x)$ where $C(x)$ is the generating function of the Catalan numbers. This number is given by $\frac{m+1}{l+1}\binom{l+1}{\frac{l-m}{2}}$ (See \cite{Lang} for a reference). All these words has $\mathrm{E}[x_{w}]=0$ and a straight forward calculation shows that
\begin{equation}
\begin{split}
    &\mathrm{Var}\left[ \left( \frac{1}{\sqrt{n}} \right)\sum_{i=1}^{n} V_{2}\left( \frac{i}{n} \right)\right]\\
    &= \sum_{m=1~|~ l-m \text{ even}} \left(\frac{m+1}{l+1}\binom{l+1}{\frac{l-m}{2}}\right)^2 \frac{\text{\# of stems of length $m$}}{n^{m+1}} (1+o(1)).
\end{split}    
\end{equation}
Further, 
\[
\text{\# of stems of length $m$} = (1+o(1)) \text{\# of paths $(i_{0},\ldots,i_{m})$ such that $i_{0},\ldots, i_{m}$ are distinct}.
\]
Hence,
\[
\mathrm{E}\left[\left( \frac{1}{\sqrt{n}} \right)\sum_{i\in\{ [cn],\ldots,[dn\}]} V_{2}\left( \frac{i}{n}\right)- \sum_{m=1~|~ l-m ~ \text{even}}^{l} \frac{m+1}{l+1}\binom{l+1}{\frac{l-m}{2}}\left[\left( \frac{1}{\sqrt{n}} \right)\sum_{i \in  \{ [cn],\ldots, [dn]\}, j \in \{ [an],\ldots, [bn]\}  } \xi^{l}_{i,j,n}\right]\right]^2 \to 0
\]
proving \eqref{eq:traceapproximationnb}.    
\end{proof}

\noindent 
The proof of Theorem \ref{Thm:wigcon} is a straightforward application of the arguments in the proof of Theorem \ref{thm:joinopconv} and Lemma \ref{lem:clt}. So we omit this.
\begin{proof}[Proof of Theorem \ref{Thm:spect}]
    First, we show that the operator $\mathcal{W}$ is bounded. Let $v_{m}$ be a Cauchy sequence of vectors in $\mathcal{H}$ such that $||v_{m}||=1$ and each $v_{m}$ can be written as a finite linear combination of $\mathbb{I}_{P_{i}}$ and $\psi_{k,i}$'s. For any fixed $m_{1},m_{2}$, we know that $||v_{n,m_{1}}- v_{n,m_{2}}|| \stackrel{p}{\to}||v_{m_{1}}-v_{m_{2}}||$ as $n \to \infty$. As the operator $\mathcal{W}$ takes finite linear combinations to finite linear combinations, we have $||K_{n}(V_{n,m_{1}})-K_{n}(V_{n,m_{2}})|| \stackrel{p}{\to} ||\mathcal{W}(v_{m_{1}})-\mathcal{W}(v_{m_{2}})||$. On the other hand, $\mathbb{P}\left[ ||K_{n}(V_{n,m_{1}})-K_{n}(V_{n,m_{2}})|| \le (2.0001) ||v_{n,m_{1}}-v_{n,m_{2}}|| \right]\to 1$. Hence, we get that $||\mathcal{W}(v_{m_{1}})-\mathcal{W}(v_{m_{2}})|| \le (2.0001)||v_{m_{1}}-v_{m_{2}}||$ with probability $1$. Hence, $\mathcal{W}(v_{m})$ is also Cauchy. Also, $\mathbb{P}\left[ ||K_{n}(v_{n,m})||\le 2.0001||v_{n,m}||  \right] \to 1$. Hence $||\mathcal{W}(v_{m})||\le 2.0001 ||v_{m}||$ with probability $1$. Hence, $\lim_{n\to \infty}||\mathcal{W}v_{m}||\le 2.0001 \lim_{m \to \infty} ||v_{m}||$ with probability $1$. So, $\mathcal{W}$ is bounded.

    \noindent 
    The proof of $\mathcal{W}$ is self-adjoint, follows from a similar argument and the fact that $K_{n}$ is self-adjoint.

    \noindent 
    Finally, observe that $\mathcal{W}^{l}(\mathbb{I}_{P_{i}})= \beta_{l}\mathbb{I}_{p_{i}}+ \psi_{l,i}.$ Hence $\langle \mathbb{I}_{P_{i}}, \mathcal{W}^{l} \mathbb{I}_{P_{i}}  \rangle = \beta_{l}\langle \mathbb{I}_{P_{i}}, \mathbb{I}_{P_{i}}\rangle + \langle \mathbb{I}_{P_{i}}, \psi_{l,i}\rangle= \beta_{l}\langle \mathbb{I}_{P_{i}}, \mathbb{I}_{P_{i}}\rangle$ with probability $1$. Since $\mathbb{I}_{P_{i}}$ is a basis of $L^{2}[0,1]$ and $\mathcal{W}$ is bounded, the spectral measure $\mu_{f}$ is almost surely the semicircular law for any $f \in L^2[0,1]$.
\end{proof}
\textbf{Acknowledgements:} I am grateful to Prof. Rajendra Bhatia for his interest and comments. I also thank Prof. Krishna Maddaly for discussions at the earlier stage of this work. 
\bibliographystyle{alpha}
\bibliography{PAR_SPI}

\end{document}